\RequirePackage{fix-cm}
\documentclass[11pt, a4paper, oneside, DIV11, final, pagebackref]{amsart}
\usepackage[notref,notcite]{showkeys}
\usepackage[utf8]{inputenc}
\usepackage[T1]{fontenc}
\usepackage{amssymb}
\usepackage[stretch=10]{microtype}
\usepackage{mathtools}
\usepackage[DIV11, headinclude=true]{typearea}
\usepackage[hidelinks, linktoc=all]{hyperref}

\providecommand{\href}[2]{#2}
\providecommand{\texorpdfstring}[2]{#1}
\providecommand*{\backref}{}
\providecommand*{\backrefalt}{}
\renewcommand*{\backref}[1]{}
\renewcommand*{\backrefalt}[4]{%
	\ifcase #1 %
	\or
	  Cited page~#2.
	\else
	  Cited pages~#2.
	\fi
}

\newcommand{\boL}{\mathcal{L}}
\newcommand{\boM}{\mathcal{M}}
\newcommand{\boN}{\mathcal{N}}
\newcommand{\boV}{\mathcal{V}}
\newcommand{\Pbb}{\mathbb{P}}
\newcommand{\E}{\mathbb{E}}
\newcommand{\Z}{\mathbb{Z}}
\newcommand{\N}{\mathbb{N}}
\newcommand{\R}{\mathbb{R}}

\newcommand{\dd}{\mathop{}\!\mathrm{d}}

\DeclarePairedDelimiter{\abs}{\lvert}{\rvert}
\DeclarePairedDelimiter{\pare}{(}{)}
\DeclarePairedDelimiter{\norm}{\lVert}{\rVert}
\newcommand{\st}{\::\:}

\def\be{\begin{equation}}
\def\ee{\end{equation}}

\def\E0{E^{(0)}}

\DeclareMathOperator{\Card}{Card}
\DeclareMathOperator{\Id}{Id}

\renewcommand{\epsilon}{\varepsilon}

\renewcommand{\phi}{\varphi}
\renewcommand{\leq}{\leqslant}
\renewcommand{\geq}{\geqslant}
\newcommand{\df}{\mathbf{d}}
\newcommand{\loc}{\mathrm{loc}}

\newtheorem{thm}{Theorem}[section]
\newtheorem{prop}[thm]{Proposition}
\newtheorem{definition}[thm]{Definition}
\newtheorem{lem}[thm]{Lemma}
\newtheorem{cor}[thm]{Corollary}
\newtheorem*{prop*}{Proposition}

\theoremstyle{definition}

\newtheorem{rmk}[thm]{Remark}

\numberwithin{equation}{section}

\title{Quantitative Pesin theory for Anosov diffeomorphisms and flows}
\author{Sébastien Gouëzel and Luchezar Stoyanov}

\address{Laboratoire Jean Leray, CNRS UMR 6629,
Université de Nantes, 2 rue de la
Houssinière,
44322 Nantes, France}
\email{sebastien.gouezel@univ-nantes.fr}

\address{School of Mathematics and Statistics, University of Western Australia, Crawley 6009 WA, Australia}
\email{luchezar.stoyanov@uwa.edu.au}

\date{\today}

\begin{document}

\begin{abstract}
Pesin sets are measurable sets along which the behavior of a matrix cocycle
above a measure preserving dynamical system is explicitly controlled. In
uniformly hyperbolic dynamics, we study how often points return to Pesin
sets under suitable conditions on the cocycle: if it is locally constant,
or if it admits invariant holonomies and is pinching and twisting, we show
that the measure of points that do not return a linear number of times to
Pesin sets is exponentially small. We discuss applications to the
exponential mixing of contact Anosov flows, and counterexamples
illustrating the necessity of suitable conditions on the cocycle.
\end{abstract}

\maketitle

\section{Introduction and main results}

Uniformly hyperbolic dynamical systems are very well understood. An approach
to study more general systems is to see to what extent they resemble
uniformly hyperbolic ones. A very fruitful approach in this respect is the
development of Pesin theory, that requires hyperbolic features (no zero
Lyapunov exponent) almost everywhere with respect to an invariant measure,
and constructs from these local stable and unstable manifolds, then leading
to results such as the ergodicity of the system under study.

A basic tool in Pesin theory is the notion of Pesin sets, made of points for
which, along their orbits, the Oseledets decomposition is well controlled in
a quantitative way. Their existence follows from general measure theory
argument, but they are not really explicit. Even in uniformly hyperbolic
situations, Pesin sets are relevant objects as the control of the Oseledets
decomposition gives directions in which the dynamics is close to conformal.
In particular, the second author has shown in~\cite{stoyanov_dolgopyat} that
Pesin sets could be used, in contact Anosov flows, to study the decay of
correlations: he proved that, if points return exponentially fast to Pesin
sets, then the correlations decay exponentially fast.

Our goal in this article is to investigate this question, for Anosov
diffeomorphisms and flows. We do not have a complete answer, but our results
indicate a dichotomy: if the dynamics is not too far away from conformality
(for instance in the case of the geodesic flow on a $1/4$-pinched compact
manifold of negative curvature), points return exponentially fast to Pesin
sets for generic metrics (in a very strong sense), and possibly for all
metrics. On the other hand, far away from conformality, this should not be
the case (we have a counter-example in a related setting, but with weaker
regularity).

Such statements are related to large deviations estimates for matrix
cocycles, i.e., products of matrices governed by the dynamics (for Pesin
theory, the cocycle is simply the differential of the map). Indeed, we will
show that such large deviations estimates make it possible to control the
returns to Pesin sets, by quantifying carefully some arguments in the proof
of Oseledets theorem.

\bigskip

Let $T: X \to X$ be a measurable map on a space $X$, preserving an ergodic
probability measure $\mu$. Consider a measurable bundle $E$ over $X$, where
each fiber is isomorphic to $\R^d$ and endowed with a norm. A linear cocycle
is a measurable map $M$ on $E$, mapping the fiber above $x$ to the fiber
above $Tx$ in a linear way, through a matrix $M(x)$. We say that the cocycle
is log-integrable if $\int \log \max(\norm{M(x)}, \norm{M(x)^{-1}}) \dd\mu(x)
< \infty$. In this case, it follows from Kingman's theorem that one can
define the Lyapunov exponents of the cocycle, denoted by $\lambda_1 \geq
\lambda_2 \geq \dotsb \geq \lambda_d$. They are the growth rate of vectors
under iteration of the cocycle, above $\mu$-almost every point. The sum
$\lambda_1+\dotsb+\lambda_i$ is also the asymptotic exponential growth rate
of the norm of the $i$-th exterior power $\Lambda^i M^n(x)$, for $\mu$-almost
every $x$.

The main condition to get exponential returns to Pesin sets is an exponential
large deviations condition.

\begin{definition}
\label{def: exp_dev_general} Consider a transformation $T$ preserving a
probability measure $\mu$, and a family of functions $u_n : X \to \R$. Assume
that, almost everywhere, $u_n(x)/n$ converges to a limit $\lambda$. We say
that the family has exponential large deviations if, for any $\epsilon>0$,
there exists $C>0$ such that, for all $n\geq 0$,
\begin{equation*}
  \mu\{  x \st \abs{u_n(x) - n \lambda} \geq n \epsilon\} \leq C e^{-C^{-1}n}.
\end{equation*}
\end{definition}

This general definition specializes to several situations that will be
relevant in this paper:
\begin{definition}
\label{def:exp_large_dev_u} Consider an integrable function $u$ above an
ergodic transformation $(T,\mu)$. We say that $u$ has exponential large
deviations if its Birkhoff sums $S_n u$ have exponential large deviations in
the sense of Definition~\ref{def: exp_dev_general}, i.e., for any
$\epsilon>0$, there exists $C>0$ such that, for all $n\geq 0$,
\begin{equation*}
  \mu\{  x \st \abs{S_n u(x) - n \int u} \geq n \epsilon\} \leq C e^{-C^{-1}n}.
\end{equation*}
\end{definition}

\begin{definition}
\label{def:exp_large_dev} Consider a log-integrable linear cocycle $M$ above
a transformation $(T,\mu)$, with Lyapunov exponents $\lambda_1 \geq \dotsb
\geq \lambda_d$. We say that $M$ has exponential large deviations for its top
exponent if the family of functions $u_n(x) = \log \norm{M^n(x)}$ (which
satisfies $u_n(x)/n \to \lambda_1$ almost everywhere) has exponential large
deviations in the sense of Definition~\ref{def: exp_dev_general}, i.e., for
any $\epsilon>0$, there exists $C>0$ such that, for all $n\geq 0$,
\begin{equation*}
  \mu\{  x \st \abs{\log \norm{M^n(x)} - n \lambda_1} \geq n \epsilon\} \leq C e^{-C^{-1}n}.
\end{equation*}
We say that $M$ has exponential large deviations for all exponents if, for
any $i\leq d$, the functions $\log \norm{\Lambda^i M^n(x)}$ satisfy
exponential large deviations in the sense of Definition~\ref{def:
exp_dev_general}, i.e., for any $\epsilon>0$, there exists $C>0$ such that,
for all $n\geq 0$,
\begin{equation}
\label{eq:exp_dev_all_exp}
  \mu\{  x \st \abs{\log \norm{\Lambda^i M^n(x)}
     - n (\lambda_1+\dotsb+\lambda_i)} \geq n \epsilon\} \leq C e^{-C^{-1}n}.
\end{equation}
\end{definition}

We will explain in the next paragraph that many linear cocycles above
subshifts of finite type have exponential large deviations for all exponents,
see Theorem~\ref{thm:large_deviations} below. This builds on techniques
developed by Bonatti, Viana and Avila (see~\cite{bonatti_viana_lyapunov,
avila_viana_criterion}). The main novelty of our work is the proof that such
large deviations imply exponential returns to Pesin sets, as we explain in
Paragraph~\ref{subsec:quantitative_Pesin}. The last paragraph of this
introduction discusses consequences of these results.

\subsection{Sufficient conditions for large deviations for linear cocycles}

In this paragraph, we consider a (bilateral) transitive subshift of finite
type $T:\Sigma \to \Sigma$, together with a Gibbs measure $\mu$ for a Hölder
potential. Let $E$ be a continuous $\R^d$-bundle over $\Sigma$, endowed with
a continuous linear cocycle $M$ on $E$ over $T$. For instance, one may take
$E=\Sigma \times \R^d$, then $M(x)$ is simply an invertible $d\times d$
matrix depending continuously on $x$. We describe in
Theorem~\ref{thm:large_deviations} various conditions under which such a
cocycle has exponential large deviations for all exponents, in the sense of
Definition~\ref{def:exp_large_dev}. Through the usual coding process, similar
results follow for hyperbolic basic sets of diffeomorphisms, and in
particular for Anosov or Axiom A diffeomorphisms.

We show in Appendix~\ref{app:counter} the existence of a continuous linear
cocycle above a subshift of finite type which does not have exponential large
deviations for its top exponent. Hence, additional assumptions are needed for
this class of results (contrary to the case of Birkhoff sums, where all
Birkhoff sums of continuous functions over a transitive subshift of finite
type have exponential large deviations). These assumptions, as is usual in
the study of linear cocycles, are defined in terms of holonomies. In a
geometric context, holonomies are usually generated by connections. In the
totally disconnected context of subshifts of finite type, connections do not
make sense, but the global notion of holonomy does.

The local stable set of $x$ is the set $W_{\loc}^s(x) = \{y \st y_n=x_n
\text{ for all } n\geq 0\}$. In the same way, its local unstable set is
$W_{\loc}^u(x) = \{y \st y_n=x_n \text{ for all } n\leq 0\}$. By definition,
$W_{\loc}^s(x) \cap W_{\loc}^u(x) = \{x\}$.

An unstable holonomy is a family of isomorphisms $H^u_{ x \to  y}$ from
$E(x)$ to $E(y)$, defined for all $x$ and $y$ with $y \in W_{\loc}^u(x)$. We
require the compatibility conditions $H^u_{x\to x} = \Id$ and $H^u_{ y \to z}
\circ H^u_{ x \to y} = H^u_{ x \to  z}$ for any $x$, $y$ and $z$ on the same
local unstable set. Moreover, we require the continuity of $(x,  y) \mapsto
H^u_{ x \to  y}$ (globally, i.e., not only along each leaf).

In the same way, one defines a stable holonomy as a family of maps $H^s_{ x
\to y}$ from $E(x)$ to $E(y)$ when $x$ and $y$ belong to the same local
stable set, with the same equivariance and continuity requirements as above.

\begin{definition}
A linear cocycle admits invariant continuous holonomies if there exist two
stable and unstable continuous holonomies, denoted respectively by $H^s$ and
$H^u$, that are equivariant with respect to the cocycle action. More
precisely, for any $x$, for any $y \in W_{\loc}^s(x)$, and any $v \in  E(x)$,
one should have
\begin{equation*}
   M(y) H^s_{ x \to  y} v = H^s_{ T  x \to  T  y}  M(x) v.
\end{equation*}
Similarly, for any $x$, for any $y \in W_{\loc}^u(x)$, and any $v \in E(x)$,
one should have
\begin{equation*}
   M(y)^{-1} H^u_{ x \to  y} v
    = H^u_{ T^{-1}  x \to  T^{-1}  y}  M(x)^{-1} v.
\end{equation*}
\end{definition}

Stable holonomies give a canonical way to trivialize the bundle over local
stable sets. Thus, to trivialize the whole bundle, one may choose an
arbitrary trivialization over an arbitrary local unstable set, and then
extend it to the whole space using the holonomies along the local stable
sets. In this trivialization, the cocycle is constant along local stable
sets, i.e., it only depends on future coordinates. Symmetrically, one can
trivialize the bundle first along a stable set, and then using unstable
holonomies along the local unstable sets. In this trivialization, the cocycle
is constant along unstable sets, and depends only on past coordinates. Note
that these two trivializations do not coincide in general, unless the stable
and unstable holonomies commute: In this case, the cocycle only depends on
the coordinate $x_0$ in the resulting trivialization, i.e., it is locally
constant. Conversely, a locally constant cocycle admits the identity as
stable and unstable invariant commuting holonomies.

We say that a linear cocycle is \emph{pinching and twisting in the sense of
Avila-Viana}~\cite{avila_viana_criterion} if it has invariant continuous
holonomies, and if there exist a periodic point $p$ (of some period $k$) and
a point $q$ which is asymptotic to $p$ both in the past and in the future
(i.e., $q\in W_{\loc}^u(p)$ and $T^i q \in W_{\loc}^s(p)$ for some $i$ which
is a multiple of $k$), such that
\begin{itemize}
\item All the eigenvalues of $M^k(p)$ are real and different.
\item Define a map $\Psi:H^s_{T^i q \to p} \circ M^i(q) \circ H^u_{p \to
    q}$ from $E(p)$ to itself. Then, for any subspaces $U$ and $V$ of
    $E(p)$ which are invariant under $M^k(p)$ (i.e., which are union of
    eigenspaces) with $\dim U + \dim V = \dim E$, then $\Psi(U) \cap V =
    \{0\}$. In other words, the map $\Psi$ puts the eigenspaces of $M^k(p)$
    in general position.
\end{itemize}
This condition ensures that the Lyapunov spectrum of any Gibbs measure is
simple, by the main result of~\cite{avila_viana_criterion}. In the space of
fiber-bunched cocycles (which automatically admit invariant continuous
holonomies), this condition is open (this is clear) and dense (this is harder
as there might be pairs of complex conjugate eigenvalues at some periodic
points, which need more work to be destroyed, see~\cite[Proposition
9.1]{bonatti_viana_lyapunov}).

\begin{thm}
\label{thm:large_deviations} Let $T$ be a transitive subshift of finite type
on a space $\Sigma$, and $\mu$ a Gibbs measure for a Hölder-continuous
potential. Consider a continuous linear cocycle $M$ on a vector bundle $E$
above $T$. Then $M$ has exponential large deviations for all exponents in the
following situations:
\begin{enumerate}
\item If all its Lyapunov exponents coincide.
\item If there us a continuous decomposition of $E$ as a direct sum of
    subbundles $E=E_1 \oplus \dotsb \oplus E_k$ which is invariant under
    $M$, such that the restriction of $M$ to each $E_i$ has exponential
    large deviations for all exponents.
\item More generally, if there is an invariant continuous flag
    decomposition $\{0\}=F_0 \subseteq F_1 \subseteq \dotsb \subseteq F_k =
    E$, such that the cocycle induced by $M$ on each $F_i/F_{i-1}$ has
    exponential large deviations for all exponents.
\item If the cocycle $M$ is locally constant in some trivialization of the
    bundle $E$ (this is equivalent to the existence of invariant continuous
    holonomies which are commuting).
\item If the cocycle $M$ admits invariant continuous holonomies, and if it
    is pinching and twisting in the sense of Avila-Viana.
\item If the cocycle $M$ admits invariant continuous holonomies, and the
    bundle is $2$-dimensional.
\end{enumerate}
\end{thm}

The first three points are easy, the interesting ones are the last ones. The
various statements can be combined to obtain other results. For instance, if
each (a priori only measurable) Oseledets subspace is in fact continuous (for
instance if the Oseledets decomposition is dominated), then the cocycle has
exponential large deviations for all exponents: this follows from points (1)
and (2) in the theorem. We expected that our techniques would show a result
containing (4--6): if a cocycle admits invariant continuous holonomies, then
it should have exponential large deviations for all exponents. However, there
is a difficulty here, see Remark~\ref{rmk:not_locally_constant}. Points (1-3)
are proved on Page~\pageref{proof:thm_1-3}, (4) on Page~\pageref{proof:4},
(5) on Page~\pageref{proof:5} and (6) on Page~\pageref{proof:6}. The proofs
of (4), (5) and (6) follow the same strategy, we will insist mainly on (4)
and indicate more quickly the modifications for (5) and (6). These proofs are
essentially applications of the techniques in~\cite{bonatti_viana_lyapunov,
avila_viana_criterion}.

\begin{rmk}
In Theorem~\ref{thm:large_deviations}, exponential large deviations are
expressed in terms of matrix norms: one should choose on each $E(x)$ a norm,
depending continuously on $x$, and then $\norm{M^n(x)}$ is the operator norm
of $M^n(x)$ between the two normed vector spaces $E(x)$ and $E(T^n x)$. The
above statement does not depend on the choice of the norm (just as the value
of the Lyapunov exponents) as the ratio between two such norms is bounded
from above and from below by compactness. Hence, we may choose whatever norm
we like most on $E$. For definiteness, we use a Euclidean norm.
\end{rmk}

The above theorem shows that, in most usual topologies, generic linear
cocycles have exponential large deviations for all exponents. Indeed, for
generic cocycles in the $C^0$ topology, the Oseledets decomposition is
dominated (see~\cite[Theorem 9.18]{viana_lyapunov}), hence (1) and (2) in the
theorem yield exponential large deviations. For generic cocycles in the
Hölder topology among fiber bunched cocycles (the most tractable Hölder
cocycles), pinching and twisting are generic, hence (5) also gives
exponential large deviations.

\subsection{Quantitative Pesin theory from large deviations for linear cocycles}

\label{subsec:quantitative_Pesin}

Let $T$ be an invertible continuous map on a compact metric space $X$,
preserving an ergodic probability measure $\mu$. Let $M$ be a continuous
cocycle above $T$, on the trivial bundle $X \times \R^d$. Denote by
$\lambda_1 \geq \dotsb \geq \lambda_d$ its Lyapunov exponents, and $I=\{i \st
\lambda_i <\lambda_{i-1}\}$. Then $(\lambda_i)_{i \in I}$ are the distinct
Lyapunov exponents. Denote by $E_i$ the corresponding Oseledets subspace, its
dimension $d_i$ is $\Card\{j\in [1,d] \st \lambda_j = \lambda_i\}$. The
subspaces $E_i(x)$ are well-defined on an invariant subset $X'$ of $X$ with
$\mu(X') = 1$ and  $E_i(T(x)) = E_i(x)$ for all $x\in X'$. Moreover
$\frac{1}{n} \log \norm{M^n(x) v} \to \lambda_i$ as $n \to \pm \infty$ for
all $v\in E_i(x)\setminus \{ 0\}$. With this notation, the space $E_i(x)$ is
repeated $d_i$ times. The distinct Oseledets subspaces are $(E_i(x))_{i \in
I}$.

Let $\epsilon>0$. The basic ingredient in Pesin theory is the function
\begin{equation}
\label{eq:def_A_epsilon}
\begin{split}
  A_\epsilon(x)&  = \sup_{i\in I} A_\epsilon^{(i)}(x)
  \\&
  = \sup_{i\in I} \sup_{v \in E_i(x)\setminus\{0\}} \sup_{m,n\in \Z}
    \frac{ \norm{M^n(x) v}}{\norm{M^m(x) v}} e^{-(n-m)\lambda_i} e^{-(\abs{n}
+ \abs{m}) \epsilon/2}
  \in [0,\infty].
 \end{split}
\end{equation}
This function is slowly varying, i.e.,
\begin{equation*}
  e^{-\epsilon} A_\epsilon(x) \leq A_\epsilon(Tx) \leq e^\epsilon A_\epsilon(x),
\end{equation*}
as the formulas for $x$ and $Tx$ are the same except for a shift of $1$ in
$n$ and $m$. Moreover, for all $k\in \Z$ and all $v\in E_i(x)$,
\begin{equation*}
  \norm{v}{A_\epsilon(x) e^{-\abs{k}\epsilon}} \leq
      \frac{\norm{M^k(x) v}}{e^{k \lambda_i }} \leq  \norm{v}{A_\epsilon(x)
e^{\abs{k}\epsilon}},
\end{equation*}
where one inequality follows by taking $m=0$ and $n=k$ in the definition of
$A_\epsilon$, and the other inequality by taking $m=k$ and $n=0$. The almost
sure finiteness of $A_\epsilon$ follows from Oseledets theorem.

\emph{Pesin sets} are sets of the form $\{x \st A_\epsilon(x) \leq C\}$, for
some constant $C>0$. Our main goal is to show that most points return
exponentially often to some Pesin set. This is the content of the following
theorem.

\begin{thm}
\label{thm:exp_returns_Pesin} Let $T$ be a transitive subshift of finite type
on a space $\Sigma$, and $\mu$ a Gibbs measure for a Hölder-continuous
potential. Consider a continuous linear cocycle $M$ on the trivial vector
bundle $\Sigma \times \R^d$ above $T$. Assume that $M$ has exponential large
deviations for all exponents, both in positive and negative times.

Let $\epsilon>0$ and $\delta>0$. Then there exists $C>0$ such that, for all
$n\in \N$,
\begin{equation*}
  \mu \{x \st \Card\{j \in [0, n-1] \st A_\epsilon(T^j x) > C\} \geq \delta n\}
\leq C e^{-C^{-1}n}.
\end{equation*}
\end{thm}

One difficulty in the proof of this theorem is that the function $A_\epsilon$
is defined in terms of the Lyapunov subspaces, which are only defined almost
everywhere, in a non constructive way. To get such controls, we will need to
revisit the proof of Oseledets to get more quantitative bounds, in
Section~\ref{thm:deterministic_Oseledets}, showing that an explicit control
on the differences $\pare*{\abs*{\log \norm{\Lambda^i M^n(x)} - n(\lambda_1 +
\dotsc + \lambda_i)}}_{n\in \Z}$ at some point $x$ implies an explicit
control on $A_\epsilon(x)$ in Theorem~\ref{thm:deterministic_Oseledets}.
Then, the number of returns to the Pesin sets is estimated using an abstract
result in subadditive ergodic theory, interesting in its own right,
Theorem~\ref{thm:exp_returns}. These two statements are finally combined in
Section~\ref{sec:proof_returns_Pesin} to prove
Theorem~\ref{thm:exp_returns_Pesin}.

\subsection{Applications}

In this paragraph, we describe several systems to which our results on large
deviations and exponential returns to Pesin sets apply.

First, coding any Anosov or Axiom A diffeomorphism thanks to a Markov
partition, then the above theorems apply to such maps, provided the matrix
cocycle has exponential large deviations. Hence, one needs to check the
conditions in Theorem~\ref{thm:large_deviations}.

\medskip

The main class of cocycles admitting stable and unstable holonomies is the
class of \emph{fiber bunched cocycles},
see~\cite[Definition~A.5]{avila_viana_criterion}.

A $\nu$-Hölder continuous cocycle $M$ over a hyperbolic map $T$ on a compact
space is $s$-fiber bunched if there exists $\theta \in (0,1)$ such that $d(T
x, T y) \leq \theta d(x,y)$ and $\norm{M(x)} \norm{M(y)^{-1}} \theta^\nu <
1$, for all $x,y$ on a common local stable set (or more generally if this
property holds for some iterate of the map and the cocycle). This means that
the expansion properties of the cocycle are dominated by the contraction
properties of the map $T$. This results in the fact that $M^n(y)^{-1} M^n(x)$
converges exponentially fast when $n\to \infty$, to a map which is a
continuous invariant stable holonomy,
see~\cite[Proposition~A.6]{avila_viana_criterion}

In the same way, one defines $u$-fiber bunched cocycles. Finally, a cocycle
is fiber-bunched if it is both $s$ and $u$-fiber bunched. For instance, if
$T$ and $\nu$ are fixed, then a cocycle which is close enough to the identity
in the $C^\nu$ topology is fiber bunched. Our results apply to such cocycles
if they are pinching and twisting, which is an open and dense condition among
fiber bunched cocycles.

\medskip

Our results also apply to generic cocycles in the $C^0$ topology. Indeed, the
Oseledets decomposition is then dominated (see~\cite[Theorem
9.18]{viana_lyapunov}), hence (1) and (2) in the theorem yield exponential
large deviations, and from there one deduces exponential returns to Pesin
sets by Theorem~\ref{thm:exp_returns_Pesin}.

\medskip

The main application we have in mind is to flows. The second author proves
in~\cite{stoyanov_dolgopyat} the following theorem:

\begin{thm}[Stoyanov~\cite{stoyanov_dolgopyat}]
\label{thm:stoyanov_mixing} Let $g_t$ be a contact Anosov flow on a compact
manifold $X$, with a Gibbs measure $\mu_X$.

Consider the first return map to a Markov section $T$, the corresponding
invariant measure $\mu$, and the corresponding derivative cocycle $M$, from
the tangent space of $X$ at $x$ to the tangent space of $X$ at $Tx$. Assume
that $(T, M, \mu)$ has exponential returns to Pesin sets as in the conclusion
of Theorem~\ref{thm:exp_returns_Pesin}.

Then the flow $g_t$ is exponentially mixing: there exists $C>0$ such that,
for any $C^1$ functions $u$ and $v$, for any $t\geq 0$,
\begin{equation*}
  \abs*{\int u\cdot v \circ g_t \dd\mu_X - \int u \dd\mu_X \cdot \int v \dd\mu_X}
  \leq C \norm{u}_{C^1} \norm{v}_{C^1} e^{-C^{-1} t}.
\end{equation*}
\end{thm}
By a standard approximation argument, exponential mixing for Hölder
continuous functions follows.

This statement is the main motivation to study exponential returns to Pesin
sets. We deduce from Theorem~\ref{thm:exp_returns_Pesin} the following:

\begin{thm}
\label{thm:contact} Consider a contact Anosov flow with a Gibbs measure, for
which the derivative cocycle has exponential large deviations for all
exponents. Then the flow is exponentially mixing.
\end{thm}

To apply this theorem in concrete situations, we have to check whether the
sufficient conditions of Theorem~\ref{thm:large_deviations} for exponential
large deviations hold. The main requirement is the existence of stable and
unstable holonomies. Unfortunately, we only know their existence when the
foliation is smooth:

\begin{lem}
\label{lem:holonomies} Consider a contact Anosov flow for which the stable
and unstable foliations are $C^1$. Then the derivative cocycle admits
continuous invariant holonomies with respect to the induced map on any Markov
section.
\end{lem}
\begin{proof}
It suffices to show that the flow admits continuous invariant holonomies
along weak unstable and weak stable manifolds, as they descend to the Markov
section.

We construct the holonomy along weak unstable leaves, the holonomy along weak
stable leaves being similar. Consider two points $x$ and $y$ on a weak
unstable leaf. Then the holonomy of the weak unstable foliation gives a local
$C^1$ diffeomorphism between $W^s(x)$ to $W^s(y)$, sending $x$ to $y$. The
derivative of this map is a canonical isomorphism between $E^s(x)$ and
$E^s(y)$, which is clearly equivariant under the dynamics. There is also a
canonical isomorphism between the flow directions at $x$ and $y$. What
remains to be done is to construct an equivariant isomorphism between
$E^u(x)$ and $E^u(y)$.

For this, we use the fact that the flow is a contact flow, i.e., there exists
a smooth one-form $\alpha$, invariant under the flow, with kernel $E^s \oplus
E^u$, whose derivative $\dd\alpha$ restricts to a symplectic form on $E^s
\oplus E^u$. We get a map $\phi$ from $E^s$ to $(E^u)^*$, mapping $v$ to
$\dd\alpha(v, \cdot)$. This map is one-to-one: a vector $v$ in its kernel
satisfies $\dd\alpha(v, w) = 0$ for all $w \in E^u$, and also for all $w \in
E^s$ as $E^s$ is Lagrangian. Hence, $v$ is in the kernel of $\dd\alpha$,
which is reduced to $0$ as $\dd\alpha$ is a symplectic form. As $E^s$ and
$E^u$ have the same dimension, it follows that $\phi$ is an isomorphism.

Consider now $x$ and $y$ on a weak unstable leaf. We have already constructed
a canonical isomorphism between $E^s(x)$ and $E^s(y)$. With the above
identification, this gives a canonical isomorphism between $(E^u(x))^*$ and
$(E^u(y))^*$, and therefore between $E^u(x)$ and $E^u(y)$. This
identification is equivariant under the flow, as $\alpha$ is invariant.
\end{proof}

For instance, for the geodesic flow on a compact Riemannian manifold with
negative curvature, the stable and unstable foliations are smooth if the
manifold is $3$-dimensional or the curvature is strictly $1/4$-pinched, i.e.,
the sectional curvature belongs everywhere to an interval $[-b^2, -a^2]$ with
$a^2/b^2 > 1/4$, by~\cite{hirsch_pugh_C1}. Hence, we deduce the following
corollary from Theorem~\ref{thm:large_deviations} (1), (6) and (5)
respectively:

\begin{cor}
\label{cor:exp_mixing} Consider the geodesic flow $g_t$ on a compact
riemannian manifold $X$ with negative curvature. Assume one of the following
conditions:
\begin{enumerate}
  \item $X$ is of dimension $3$.
  \item $X$ is of dimension $5$ and the curvature is strictly $1/4$ inched.
  \item $X$ has any dimension, the curvature is strictly $1/4$ pinched, and
      moreover the flow is pinching and twisting.
\end{enumerate}
Then the flow is exponentially mixing for any Gibbs measure.
\end{cor}

However, these results were already proved by the second author, under weaker
assumptions: exponential mixing holds if the curvature is (not necessarily
strictly) $1/4$-pinched, in any dimension (without twisting and pinching).
This follows from the articles~\cite{stoyanov_spectrum}, in which it is
proved that a contact Anosov flow with Lipschitz holonomies and satisfying a
geometric condition is exponentially mixing for all Gibbs measure, and
from~\cite{stoyanov_pinching} where the aforementioned geometric condition is
proved to be satisfied in a class of flows including geodesic flows when the
curvature is $1/4$-pinched.

On the opposite side, the techniques of~\cite{liverani_contact}
or~\cite{faure_tsujii_flot2} prove exponential mixing for any contact Anosov
flow, without any pinching condition, but for Lebesgue measure (or for Gibbs
measure whose potential is not too far away from the potential giving rise to
Lebesgue measure): they are never able to handle all Gibbs measure.

The hope was that Theorem~\ref{thm:stoyanov_mixing} would be able to bridge
the gap between these results and the results of Dolgopyat, proving
exponential mixing for all contact Anosov flows and all Gibbs measures.
However we still need geometric conditions on the manifold to be able to
proceed. The counterexample in the Appendix~\ref{app:counter} shows that in
general exponential large deviations do not hold. Whether one can design
similar counterexamples for nice systems, e.g.\ contact Anosov flows, remains
unknown at this stage. It is also unknown whether one can prove a result
similar to Theorem~\ref{thm:contact} without assuming exponential large
deviations for all exponents.

\section{Preliminaries}

\subsection{Oseledets theorem}

Let $A$ be a linear transformation between two Euclidean spaces of the same
dimensions. We recall that, in suitable orthonormal bases at the beginning
and at the end, $A$ can be put in diagonal form with entries $s_1 \geq \dotsb
\geq s_d \geq 0$. The $s_i$ are the \emph{singular values} of $A$. They are
also the eigenvalues of the symmetric matrix $\sqrt{A^t \cdot A}$. The
largest one $s_1$ is the norm of $A$, the smallest one $s_d$ is its smallest
expansion. The singular values of $A^{-1}$ are $1/s_d \geq \dotsb \geq
1/s_1$. For any $i \leq d$, denote by $\Lambda^i A$ the $i$-th exterior
product of $A$, given by
\begin{equation*}
  (\Lambda^i A)  (v_1 \wedge v_2 \wedge \dotsb \wedge v_i) = Av_1 \wedge Av_2 \wedge \cdots\wedge Av_i .
\end{equation*}
Then
\begin{equation*}
  \norm{\Lambda^i A} = s_1 \dotsm s_i,
\end{equation*}
as $\Lambda^i A$ is diagonal in the corresponding orthonormal bases.

\medskip

Consider a transformation $T$ of a space $X$, together with a finite
dimensional real vector bundle $E$ above $X$: all the fibers are isomorphic
to $\R^d$ for some $d$ and the bundle is locally trivial by definition. For
instance, $E$ may be the product bundle $X \times \R^d$, but general bundles
are also allowed. In our main case of interest, $T$ will be a subshift of
finite type. In this case, any such continuous vector bundle is isomorphic to
$X \times \R^d$: by compactness, there is some $N>0$ such that the bundle is
trivial on all cylinders $[x_{-N},\dotsc, x_N]$. As these (finitely many)
sets are open and closed, trivializations on these cylinders can be glued
together to form a global trivialization of the bundle. In the course of the
proof, even if we start with the trivial bundle, we will have to consider
general bundles, but they will be reducible to trivial bundles thanks to this
procedure.

A cocycle is a map $M$ associating to $x\in X$ an invertible linear operator
$M(x): E(x) \to E(T x)$ (where $E(x)$ denotes the fiber of the fiber bundle
above $x$). When $E=X \times \R^d$, then $M(x)$ is simply a $d\times d$
matrix. The iterated cocycle is given by $M^{n}(x) =  M(T^{n-1} x) \dotsm
M(x)$ for $n\geq 0$, and by $M^{-n}(x) =  M(T^{-n}  x)^{-1} \dotsm  M(T^{-1}
x)$. It maps $E(x)$ to $E(T^n x)$ in all cases. Be careful that, with this
notation, $M^{-1}(x) \neq M(x)^{-1}$: the first notation indicates the
inverse of the cocycle, with the intrinsic time shift, going from $E(x)$ to
$E(T^{-1} x)$, while the second one is the inverse of a linear operator, so
it goes from $E(T  x)$ to $E(x)$. In general, $M^{-n}(x) = M^n(T^{-n}
x)^{-1}$.

\medskip

Assume now that $T$ is invertible, that it preserves an ergodic probability
measure, and that the cocycle $M$ is $\log$-integrable. For any $i \leq d$,
the quantity $x \mapsto \log \norm{\Lambda^i (M^n(x))}$ is a subadditive
cocycle. Hence, by Kingman's theorem, $\log \norm{\Lambda^i (M^n(x))} / n$
converges almost surely to a constant quantity that we may write as
$\lambda_1 + \dotsb + \lambda_i$, for some scalars $\lambda_i$. These are
called the \emph{Lyapunov exponents} of the cocycle $M$ with respect to the
dynamics $T$ and the measure $\mu$. Let $I = \{i \st \lambda_i <
\lambda_{i-1} \}$ parameterize the distinct Lyapunov exponents, and let $d_i=
\Card\{j \st \lambda_j=\lambda_i\}$ be the multiplicity of $\lambda_i$.

In this setting, the Oseledets theorem asserts that the $\lambda_i$ are
exactly the asymptotic growth rates of vectors, at almost every point. Here
is a precise version of this statement (see for
instance~\cite[Theorem~3.4.11]{arnold_random}).

\begin{thm}[Oseledets Theorem]
\label{thm:Oseledets} Assume that the cocycle $M$ is log-integrable. Then:
\begin{enumerate}
\item For $i \in I$, define $E_i(x)$ to be the set of nonzero vectors $v
    \in E(x)$ such that, when $n \to \pm \infty$, then $\log \norm{M^n(x)
    v}/n \to \lambda_i$, to which one adds the zero vector. For
    $\mu$-almost every $x$, this is a vector subspace of $E(x)$, of
    dimension $d_i$. These subspaces satisfy $E(x) = \bigoplus_{i \in I}
    E_i(x)$. Moreover, $M(x) E_i(x) = E_i(T x)$.
\item Almost surely, for any $i\in I$, one has $\log
    \norm{M^n(x)_{|E_i(x)}} / n \to \lambda_i$ when $n \to \pm \infty$, and
    $\log \norm{M^n(x)_{|E_i(x)}^{-1}} / n \to -\lambda_i$.
\end{enumerate}
\end{thm}

In other words, the decomposition of the space $E(x) = \bigoplus_{i\in I}
E_i(x)$ gives a block-diagonal decomposition of the cocycle $M$, such that in
each block the cocycle has an asymptotic behavior given by $e^{n\lambda_i}$
up to subexponential fluctuations.

The spaces $E_i(x)$ can be constructed almost surely as follows. Let
$t_1^{(n)}(x) \geq \dotsb \geq t_d^{(n)}(x)$ be the singular values of
$M^n(x)$. They are the eigenvalues of the symmetric matrix $\sqrt{ M^n(x)^t
\cdot M^n(x)}$, the corresponding eigenspaces being orthogonal. Write
$t_i^{(n)}(x) = e^{n \lambda_i^{(n)}(x)}$. Then $\lambda_i^{(n)}(x)$
converges to $\lambda_i$ for almost every $x$. In particular, for $i \in I$,
one has $t_{i-1}^{(n)}(x) > t_i^{(n)}(x)$ for large enough $n>0$. It follows
that the direct sum of the eigenspaces of $\sqrt{ M^n(x)^t \cdot M^n(x)}$ for
the eigenvalues $t_i^{(n)}(x),\dotsc, t_{i+d_i-1}^{(n)}(x)$ is well defined.
Denote it by $F^{(n)}_i(x)$. We will write $F^{(n)}_{\geq i}$ for
$\bigoplus_{j \geq i, j\in I} F^{(n)}_j(x)$, and similarly for $F^{(n)}_{\leq
i}$. In the same way, we define similar quantities for $n<0$.

\begin{thm}
\label{thm:Oseledets_limit} Fix $i\in I$. With these notations,
$F^{(n)}_{i}(x)$ converges almost surely when $n \to \infty$, to a vector
subspace $F^{(\infty)}_i(x) \subseteq E(x)$. In the same way, $F^{(-n)}_i$
converges almost surely to a space $F^{(-\infty)}_{i}(x)$. Moreover, the
direct sums $F_{\geq i}^{(\infty)}(x)$ and $F_{\leq i}^{(-\infty)}(x)$ are
almost surely transverse, and their intersection is $E_i(x)$.
\end{thm}
See~\cite[Theorem 3.4.1 and Page 154]{arnold_random}. One can reformulate the
theorem as follows. The subspaces $F^{(n)}_{\geq i}(x)$ (which are decreasing
with $i$, i.e., they form a flag) converge when $n\to \infty$ to the flag
$E_{\geq i}(x)$. Note that $F^{(n)}_{\geq i}(x)$ is only defined in terms of
the positive times of the dynamics, hence this is also the case of $E_{\geq
i}(x)$: this is the set of vectors for which the expansion in positive time
is at most $e^{n \lambda_i}$, up to subexponential fluctuations (note that
this condition is clearly stable under addition, and therefore defines a
vector subspace, contrary to the condition that the expansion would be
bounded \emph{below} by $e^{n \lambda_i}$). In the same way, $F^{(-n)}_{\leq
i}(x)$ converges when $n \to \infty$ to $E_{\leq i}(x)$, which therefore only
depends on the past of the dynamics. On the other hand, $E_i(x)$, being
defined as the intersection of two spaces depending on positive and negative
times, depends on the whole dynamics and is therefore more difficult to
analyze. We emphasize that $E_i(x)$ is in general different from
$F^{(\infty)}_i(x)$ or $F^{(-\infty)}_i(x)$.

In the above theorem, when we mention the convergence of subspaces, we are
using the natural topology on the \emph{Grassmann manifold} of linear
subspaces of some given dimension $p$. It comes for instance from the
following distance, that we will use later on:
\begin{equation}
  \label{eq:distance_Grass}
  \df(U,V) =  \norm{\pi_{U\to V^\perp}} = \max_{u\in U, \norm{u}=1} \norm{\pi_{V^\perp} u},
\end{equation}
where $\pi_{U \to V^\perp}$ is the orthogonal projection from $U$ to the
orthogonal $V^\perp$ of $V$. It is not completely obvious that this formula
indeed defines a distance. As $\df(U, V) = \norm{\pi_{V^\perp} \pi_U}$, the
triangular inequality follows from the following computation (in which we use
that orthogonal projections have norm at most $1$):
\begin{align*}
  \df(U,W) & = \norm{\pi_{W^\perp} \pi_U} = \norm{\pi_{W^\perp} (\pi_V + \pi_{V^\perp}) \pi_U}
  \leq \norm{\pi_{W^\perp} \pi_V \pi_U} + \norm{\pi_{W^\perp} \pi_{V^\perp} \pi_U}
  \\&
  \leq \norm{\pi_{W^\perp} \pi_V} + \norm{\pi_{V^\perp} \pi_U}
  = \df(V, W) + \df(U,V).
\end{align*}
For the symmetry, we note that $\df(U,V) = \sqrt{1- \norm{\pi_{U \to
V}}_{\min}^2}$, where $\norm{M}_{\min}$ denotes the minimal expansion of a
vector by a linear map $M$. This is also its smallest singular value. As
$\pi_{V \to U} = \pi_{U \to V}^t$, and a (square) matrix and its transpose
have the same singular values, it follows that $\norm{\pi_{U \to V}}_{\min} =
\norm{\pi_{V \to U}}_{\min}$, and therefore that $\df(U, V) = \df(V, U)$.

\subsection{Oseledets decomposition and subbundles}

The following lemma follows directly from Oseledets theorem, by considering
the Oseledets decomposition in each subbundle.
\begin{lem}
\label{lem:direct_split} Consider a log-integrable cocycle $M$ on a normed
vector bundle $E$, over an ergodic probability preserving dynamical system
$T$. Assume that $E$ splits as a direct sum of invariant subbundles $E_i$.
Then the Lyapunov spectrum of $M$ on $E$ is the union of the Lyapunov spectra
of $M$ on each $E_i$, with multiplicities.
\end{lem}

The same holds if $M$, instead of leaving each $E_i$ invariant, is upper
triangular. While this is well known, we give a full proof as this is not as
trivial as one might think.
\begin{lem}
\label{lem:flag_split} Consider a log-integrable cocycle $M$ on a normed
vector bundle $E$, over an ergodic probability preserving dynamical system
$T$. Assume that there is a measurable invariant flag decomposition $\{0\} =
F_0(x) \subseteq F_1(x) \subseteq \dotsb \subseteq F_k(x) = E(x)$. Then the
Lyapunov spectrum of $M$ on $E$ is the union of the Lyapunov spectra of $M$
on each $F_i/F_{i-1}$, with multiplicities.
\end{lem}
Equivalently, considering $E_i$ a complementary subspace to $F_{i-1}$ in
$F_i$, then the matrix representation of $M$ in the decomposition $E=E_1
\oplus \dotsb \oplus E_k$ is upper triangular, and the lemma asserts that the
Lyapunov spectrum of $M$ is the union of the Lyapunov spectra of the diagonal
blocks.
\begin{proof}
Passing to the natural extension if necessary, we can assume that $T$ is
invertible.

Let us first assume that $k=2$, and that there is only one Lyapunov exponent
$\lambda$ in $E_1$ and one Lyapunov exponent $\mu$ in $E_2$, both with some
multiplicity. In matrix form, $M$ can be written as
$\left(\begin{smallmatrix} A_1 & B \\ 0 & A_2 \end{smallmatrix}\right)$,
where the growth rate of $A_1^n$ and $A_2^n$ are respectively given by
$e^{\lambda n}$ and $e^{\mu n}$. Then
\begin{equation}
\label{eq:Mn}
   M^n(x) = \left(\begin{array}{cc} A_1^n(x) & \sum_{k=1}^n A_1^{n-k}(T^k x) B(T^{k-1}x) A_2^{k-1}(x)
   \\ 0 & A_2^n(x) \end{array} \right).
\end{equation}
As $M$ is a log-integrable cocycle, $\log \norm{M(T^n x)}/n$ tends almost
surely to $0$ by Birkhoff theorem. Hence, the growth of $\norm{B(T^n x)}$ is
subexponential almost surely.

Assume first $\lambda > \mu$. Define a function $\Phi(x) : E_2(x) \to E_1(x)$
by
\begin{equation*}
  \Phi(x) = - \sum_{k=0}^\infty A_1^{k+1}(x)^{-1} B(T^k x) A_2^k(x).
\end{equation*}
The series converges almost surely as $\norm{A_1^{k+1}(x)^{-1} B(T^k x)
A_2^k(x)} \leq C  e^{(\mu-\lambda)k + \epsilon k}$ and $\mu-\lambda < 0$.
This series is designed so that $A_1(x) \Phi(x) + B(x) = \Phi(Tx) A_2(x)$,
i.e., so that the subspace $\tilde E_2(x) = \{ (\Phi(x)v, v) \st v \in
E_2(x)\}$ is invariant under $M$. We have obtained a decomposition $E= E_1
\oplus \tilde E_2$, on which the cocycle acts respectively like $A_1$ and
$A_2$. Hence, the result follows from Lemma~\ref{lem:direct_split}.

Assume now $\lambda<\mu$. Then one can solve again the equation $A_1(x)
\Phi(x) + B(x) = \Phi(Tx) A_2(x)$, this time going towards the past, by the
converging series
\begin{equation*}
  \Phi(x) = \sum_{k=0}^\infty A_1^{-k}(x)^{-1} B(T^{-k}x) A_2^{-k-1}(x).
\end{equation*}
Then, one concludes as above.

Finally, assume $\lambda=\mu$. For any typical $x$, any $n$ and any $k\leq
n$, we have
\begin{align*}
  \norm{A_1^{n-k}(T^k x)} &
  = \norm{A_1^n(x) A_1^k(x)^{-1}} \leq \norm{A_1^n(x)} \norm{A_1^k(x)^{-1}}
  \\&
  \leq C e^{\lambda n + \epsilon n/4} \cdot e^{-\lambda k + \epsilon k/4}
  \leq C e^{\lambda(n-k) + \epsilon n/2}.
\end{align*}
Hence, one deduces from the expression~\eqref{eq:Mn} of $M^n(x)$ that its
norm grows at most like $n e^{n\lambda+n\epsilon}$ almost surely. Hence, all
its Lyapunov exponents are $\leq \lambda$. The same argument applied to the
inverse cocycle, for $T^{-1}$, shows that all the Lyapunov exponents are also
$\geq \lambda$, concluding the proof in this case.

\medskip

We turn to the general case. Subdividing further each $F_i/F_{i-1}$ into the
sum of its Oseledets subspaces, we may assume that there is one single
Lyapunov exponent in each $F_i/F_{i-1}$. Then, we argue by induction over
$k$. At step $k$, the induction assumption ensures that the Lyapunov spectrum
$L_2$ of $M$ in $E/F_1$ is the union of the Lyapunov spectra in the
$F_i/F_{i-1}$ for $i>1$. Denoting by $L_1$ the Lyapunov spectrum in $F_1$
(made of a single eigenvalue $\lambda$ with some multiplicity), we want to
show that the whole Lyapunov spectrum is $L_1 \cup L_2$, with multiplicities.
Using the Oseledets theorem in $E/F_1$ and lifting the corresponding bundles
to $E$, we obtain subbundles $G_2,\dotsc, G_I$ such that, in the
decomposition $E=F_1 \oplus G_2 \oplus \dotsb \oplus G_I$, the matrix $M$ is
block diagonal, except possibly for additional blocks along the first line.
Each block $G_i$ in which the Lyapunov exponent is not $\lambda$ can be
replaced by a block $\tilde G_i$ which is really invariant under the
dynamics, as in the $k=2$ case above. We are left with $F_1$ and possibly one
single additional block, say $G_i$, with the same exponent $\lambda$. The
$k=2$ case again shows that all the Lyapunov exponents in $F_1 \oplus G_i$
are equal to $\lambda$, concluding the proof.
\end{proof}

\section{Exponential large deviations for norms of linear cocycles}

\subsection{Gibbs measures}
\label{subsec:Gibbs}

In this section, we recall basic properties of Gibbs measures, as explained
for instance in~\cite{bowen} and~\cite{parry-pollicott}. By \emph{Gibbs
measure}, we always mean in this article Gibbs measure with respect to some
Hölder continuous potential.

Let $\phi$ be a Hölder-continuous function, over a transitive subshift of
finite type $T:\Sigma \to \Sigma$. The \emph{Gibbs measure} associated to
$\phi$, denoted by $\mu_\phi$, is the unique $T$-invariant probability
measure for which there exist two constants $P$ (the \emph{pressure} of
$\phi$) and $C>0$ such that, for any cylinder $[a_0,\dotsc, a_{n-1}]$, and
for any point $x$ in this cylinder,
\begin{equation}
\label{eq:Gibbs}
  C^{-1} \leq \frac{\mu_\phi[a_0,\dotsc, a_{n-1}]}{e^{S_n \phi(x) - nP}} \leq C.
\end{equation}
The Gibbs measure only depends on $\phi$ up to the addition of a coboundary
and a constant, i.e., $\mu_{\phi} = \mu_{\phi + g-g\circ T + c}$.

Here is an efficient way to construct the Gibbs measure. Any Hölder
continuous function is cohomologous to a Hölder continuous function which
only depends on positive coordinates of points in $\Sigma$. Without loss of
generality, we can assume that this is the case of $\phi$, and also that
$P(\phi) = 0$. Denote by $T_+: \Sigma_+ \to \Sigma_+$ the unilateral subshift
corresponding to $T$. Define the transfer operator $\boL_\phi$, acting on the
space $C^\alpha$ of Hölder continuous functions on $\Sigma_+$ by
\begin{equation*}
  \boL_\phi u(x_+) = \sum_{T_+ y_+ = x_+} e^{\phi(y_+)} v(y_+).
\end{equation*}
Then one shows that this operator has a simple eigenvalue $1$ at $1$,
finitely many eigenvalues of modulus $1$ different from $1$ (they only exist
if $T$ is transitive but not mixing) and the rest of its spectrum is
contained in a disk of radius $<1$. One deduces that, for any $v \in
C^\alpha$, then in $C^\alpha$ one has $\frac{1}{N} \sum_{n=0}^{N-1}
\boL_\phi^n u \to \mu^+(v) v_0$, where $v_0$ is a (positive) eigenfunction
corresponding to the eigenvalue $1$, and $\mu^+$ is a linear form on
$C^\alpha$. One can normalize them by $\mu^+(1) = 1$. By approximation, it
follows that this convergence also holds in $C^0$ for $v\in C^0$. Moreover,
$\mu^+$ extends to a continuous linear form on $C^0$, i.e., it is a
probability measure.

Replacing $\phi$ with $\phi + \log v_0 -\log v_0 \circ T_+$, one replaces the
operator $\boL_\phi$ (with eigenfunction $v_0$) with the operator $\boL_{\phi
+ \log v_0 - \log v_0 \circ T_+}$, with eigenfunction $1$. Hence, without
loss of generality, we can assume that $v_0 = 1$. With this normalization,
one checks that the measure $\mu^+$ is $T_+$-invariant. It is the Gibbs
measure for $T_+$, satisfying the property~\eqref{eq:Gibbs}. Its natural
$T$-invariant extension $\mu$ to $\Sigma$ is the Gibbs measure for $T$. We
have for any $v \in C^0(\Sigma_+)$
\begin{equation}
\label{eq:conv_boL}
  \frac{1}{N} \sum_{n=0}^{N-1} \boL^n v (x_+) \to \int v \dd \mu^+,
     \quad \text{uniformly in $x_+ \in \Sigma_+$.}
\end{equation}

It follows from the construction above that the jacobian of $\mu^+$  with
respect to $T_+$ is given by $J(x_+) = \frac{\dd T^* \mu_+}{\mu_+}(x) =
e^{-\phi(x_+)}$.

Consider the disintegration of $\mu$ with respect to the factor $\mu^+$:
there exists a family of measures $\mu^-_{x_+}$ on $W_{\loc}^s(x_+)$ for $x_+
\in \Sigma_+$, such that $\mu = \int \mu^-_{x_+} \dd\mu_+(x_+)$. Formally, we
write $\mu = \mu_+ \otimes \mu_{x_+}^-$, even though this is not a direct
product. These measures can in fact be defined for all $x_+$ (instead of
almost all $x_+$) in a canonical way, they depend continuously on $x_+$, they
belong to the same measure class when the first coordinate $(x_+)_0$ is
fixed, and moreover their respective Radon-Nikodym derivatives are continuous
in all variables. See for instance~\cite[Section A.2]{avila_viana_criterion}.

Geometrically, the picture is the following. Consider some point $x_+ \in
\Sigma_+$. It has finitely many preimages $y_+^1,\dotsc, y_+^i$ under $T_+$.
Then $W_{\loc}^s(x_+) = \bigcup_i T(W_{\loc}^s(y_+^i))$, and
\begin{equation}
\label{eq:T-inv}
  \mu_{x_+}^- = \sum_i \frac{1}{J(y_+^i)} T_* \mu_{y_+^i}^- = \sum_i e^{\phi(y_+^i)} T_* \mu_{y_+^i}^-.
\end{equation}

\subsection{First easy bounds}

In this paragraph, we prove (1-3) in Theorem~\ref{thm:large_deviations}.

\begin{lem}
\label{lem:anx_bound_Birkhoff} Let $a_n(x)=a(n, x)$ be a subadditive cocycle
which is bounded in absolute value for any $n$. Then, for any $N$, there
exists $C>0$ with
\begin{equation*}
a(n, x) \leq S_n (a_N/N) (x) + C,
\end{equation*}
for any $n\in \N$ and any $x \in  \Sigma$.
\end{lem}
\begin{proof}
This is clear for $n \leq 2N$ as all those quantities are bounded. Consider
now $n \geq 2N$, consider $p$ such that $n=Np+r$ with $r\in [N, 2N]$. For any
$j\in [0,N-1]$, one may write $n=j+Np+r$ with $r\in [0,2N]$. Thus,
\begin{equation*}
  a(n,  x) \leq a(j,  x) + \sum_{i=0}^{p-1} a(N,  T^{iN+j}  x) + a (r, T^{pN+j} x)
  \leq C + \sum_{i=0}^{p-1} N (a_N/N) (T^{iN+j}  x).
\end{equation*}
Summing over $j\in [0,N-1]$, we get
\begin{align*}
  N a(n, x) & \leq NC + \sum_{j=0}^{N-1} \sum_{i=0}^{p-1} N (a_N/N)(T^{iN+j} x)
  = NC + N S_{Np} (a_N/N)(x)
  \\& \leq C' + N S_n (a_N/N)(x).
\end{align*}
This proves the claim.
\end{proof}

\begin{lem}
\label{lem:upper_anx} Let $(T,\mu)$ be a transitive subshift of finite type
with a Gibbs measure, and $a(n,x)$ a subadditive cocycle above $T$ such that
$a(n,\cdot)$ is continuous for all $n$. Let $\lambda$ be the almost sure
limit of $a(n,x)/n$, assume $\lambda>-\infty$. Then, for any $\epsilon>0$,
there exists $C>0$ such that, for all $n\geq 0$,
\begin{equation*}
  \mu\{ x \st a(n,x) \geq n \lambda + n \epsilon\} \leq C e^{-C^{-1}n}.
\end{equation*}
\end{lem}
\begin{proof}
By Kingman's theorem, $a(n, x)/n$ converges to $\lambda$ almost everywhere
and in $L^1$. Thus, one can take $N$ such that $\int a_N/N \dd\mu(x) \leq
(\lambda + \epsilon/2) N$. From the previous lemma, we obtain a constant $C$
such that, for all $n$ and $x$,
\begin{equation*}
  a(n, x) \leq S_n(a_N/N)  x + C.
\end{equation*}
Thus,
\begin{equation*}
  \{ x \st a(n, x) \geq n \lambda + n \epsilon\}
  \subseteq \{ x \st S_n(a_N/N)  x \geq n \int(a_N/N) + n\epsilon/2-C\}.
\end{equation*}
By the large deviations inequality for continuous functions\footnote{This
holds for continuous functions in transitive subshifts of finite type, by
reduction to the mixing setting after taking a finite iterate of the map, and
by reduction to Hölder continuous functions by uniform approximation.}, this
set has exponentially small measure. This proves the lemma.
\end{proof}

\begin{prop}
\label{prop:upper} Let $(T,\mu)$ be a transitive subshift of finite type with
a Gibbs measure, and $M$ a continuous cocycle above $T$ with Lyapunov
exponents $\lambda_1\geq \dotsb \geq \lambda_d$. For any $\epsilon>0$, there
exists $C>0$ such that, for all $n\geq 0$ and all $i \leq d$,
\begin{equation*}
  \mu\{  x \st \log \norm{\Lambda^i M^n(x)}
     \geq n (\lambda_1+\dotsb+\lambda_i) + n \epsilon\} \leq C e^{-C^{-1}n}.
\end{equation*}
\end{prop}
\begin{proof}
Fix $i\leq d$. Then the result follows from the previous lemma applied to
$a(n, x) = \log \norm{ \Lambda^i M^n(x)}$.
\end{proof}

This proposition shows one of the two directions in
Theorem~\ref{thm:large_deviations}, without any assumption on the cocycle.
Hence, to prove this theorem, it will suffice to prove the corresponding
lower bound
\begin{equation}
\label{eq:lower_bound}
  \mu \{ x \st \log \norm{ \Lambda^i M^n(x)}
    \leq n (\lambda_1+\dotsb+\lambda_i) - n \epsilon\} \leq C e^{-C^{-1}n},
\end{equation}
under the various possible assumptions of this theorem. As is usual with
subadditive ergodic theory, this lower bound is significantly harder than the
upper bound. Indeed, the analogue of Lemma~\ref{lem:upper_anx} for the lower
bound is false, see Proposition~\ref{prop:counter_subadd} in
Appendix~\ref{app:counter}

We already have enough tools to prove the easy cases of
Theorem~\ref{thm:large_deviations}.

\begin{proof}[Proof of Theorem~\ref{thm:large_deviations} (1-3)]
\label{proof:thm_1-3} First, we prove (1): assuming that
$\lambda_1=\dotsc=\lambda_d=\lambda$, we have to prove that
\begin{equation*}
\mu \{ x \st \log \norm{ \Lambda^i M^n(x)}
    \leq n i\lambda - n \epsilon\} \leq C e^{-C^{-1}n},
\end{equation*}
Let $s_i(x,n)$ be the $i$-th singular value of $M^n(x)$. Then
\begin{equation*}
  \norm{\Lambda^i M^n(x)} = s_1(x,n) \dotsm s_i(x,n) \geq s_d(x,n)^i =
\norm{M^n(x)^{-1}}^{-i}.
\end{equation*}
Hence, to conclude, it suffices to show that
\begin{equation*}
\mu \{ x \st \log \norm{M^n(x)^{-1}}
    \geq -n\lambda + n \epsilon\} \leq C e^{-C^{-1}n}.
\end{equation*}
This follows from Proposition~\ref{prop:upper} applied to the cocycle $\tilde
M(x) = (M(x)^{-1})^t$, whose Lyapunov exponents are all equal to $-\lambda$.

\medskip

Let us now prove (3), for $k=2$ as the general case then follows by induction
over $k$. Assume that $E_1$ is an invariant continuous subbundle such that,
on $E_1$ and on $E/E_1$, the induced cocycle has exponential large deviations
for all exponents. Denote by $L_1$ and $L_2$ the Lyapunov exponents of the
cocycle on these two bundles, then the Lyapunov spectrum on $E$ is $L_1 \cup
L_2$ with multiplicity, by Lemma~\ref{lem:flag_split}. Let $E_2$ be the
orthogonal complement to $E_1$. We want to show~\eqref{eq:lower_bound}, for
some $i$. In $\lambda_1,\dotsc,\lambda_i$, some of these exponents, say a
number $i_1$ of them, are the top exponents in $L_1$. Denote their sum by
$\Sigma_1$. The remaining $i_2 = i-i_1$ exponents are the top exponents in
$L_2$, and add up to a number $\Sigma_2$.

In the decomposition $E=E_1 \oplus E_2$, the matrix $M$ is block diagonal, of
the form $\left(\begin{smallmatrix} M_1 & B \\ 0 & M_2
\end{smallmatrix}\right)$. One has $\norm{\Lambda^i M(x)} \geq
\norm{\Lambda^{i_1} M_1(x)} \norm{\Lambda^{i_2} M_2(x)}$: considering $v_1$
and $v_2$ that are maximally expanded by $\Lambda^{i_1} M_1(x)$ and
$\Lambda^{i_2} M_2(x)$, the expansion factor of $\Lambda^i M(x)$ along $v_1
\wedge v_2$ is at least $\norm{\Lambda^{i_1} M_1(x)} \norm{\Lambda^{i_2}
M_2(x)}$ thanks to the orthogonality of $E_1$ and $E_2$, and the
block-diagonal form of $M(x)$. Therefore,
\begin{align*}
  \{ x \st \log & \norm{ \Lambda^i M^n(x)}
    \leq n (\lambda_1+\dotsb+\lambda_i) - n \epsilon\}
  \\&
  \subseteq \{ x \st \log \norm{\Lambda^{i_1} M_1^n(x)} + \log \norm{\Lambda^{i_2} M_2^n(x)}
    \leq n \Sigma_1 + n\Sigma_2 - n \epsilon\}
  \\&
  \subseteq \{ x \st \log \norm{\Lambda^{i_1} M_1^n(x)} \leq n \Sigma_1 - n \epsilon/2\}
  \cup \{ x \st \log \norm{\Lambda^{i_2} M_2^n(x)} \leq n \Sigma_2 - n \epsilon/2\}.
\end{align*}
The last sets both have an exponentially small measure, as we are assuming
that the induced cocycles on $E_1$ and $E/E_1$ have exponential large
deviations for all exponents. Hence, $\mu \{ x \st \log \norm{ \Lambda^i
M^n(x)} \leq n (\lambda_1+\dotsb+\lambda_i) - n \epsilon\}$ is also
exponentially small. This concludes the proof of (3).

\medskip

Finally, (2) follows from (3) by taking $F_i=E_1 \oplus \dotsb \oplus E_i$.
\end{proof}

\subsection{\texorpdfstring{$u$}{u}-states}

Consider a cocycle $M$ admitting invariant continuous holonomies. We define a
fibered dynamics over the projective bundle $\Pbb(E)$ by
\begin{equation*}
   T_{\Pbb} (x, [v]) = (Tx, [M(x) v]).
\end{equation*}
Let $\pi_{\Pbb(E) \to  \Sigma} : \Pbb(E) \to  \Sigma$ be the first
projection.

In general, $T_{\Pbb}$ admits many invariant measures which project under
$\pi_{\Pbb(E) \to  \Sigma}$ to a given Gibbs measure $\mu$. For instance, if
the Lyapunov spectrum of $M$ is simple, denote by $v_i(x)$ the vector in
$E(x)$ corresponding to the $i$-th Lyapunov exponent, then $\mu \otimes
\delta_{[v_i(x)]}$ is invariant under $T_{\Pbb}$. By this notation, we mean
the measure such that, for any continuous function $f$,
\begin{equation*}
  \int f(x,v) \dd(\mu \otimes \delta_{[v_i(x)]}) (x,v)
  = \int f(x, [v_i(x)]) \dd \mu(x).
\end{equation*}
More generally, if $m_{ x}$ is a family of measures on $\Pbb(E(x))$ depending
measurably on $x$ such that $M(x)_*  m_{ x} = m_{ T x}$, then the measure
$\mu \otimes  m_{ x}$ (defined as above) is invariant under $T_{\Pbb}$.
Conversely, any $T_{\Pbb}$-invariant measure that projects down to $\mu$ can
be written in this form, by Rokhlin's disintegration theorem.

To understand the growth of the norm of the cocycle, we need to distinguish
among those measures the one that corresponds to the maximal expansion, i.e.,
$\mu \otimes \delta_{[v_1]}$. This measure can be obtained as follows,
assuming that $\lambda_1$ is simple. Start from a measure on $\Pbb(E)$ that
is of the form $\mu \otimes \nu_{ x}$ where the measures $\nu_{ x}$ depend
continuously on $x$ and give zero mass to all hyperplanes. Then
\begin{equation*}
   (T_{\Pbb}^n)_* (\mu \otimes \nu_{ x}) =  \mu \otimes (M^n(T^{-n}  x)_* \nu_{ T^{-n}  x}).
\end{equation*}
By Oseledets theorem, the matrix $M^n(T^{-n} x)$ acts as a contraction on
$\Pbb(E(T^{-n} x))$, sending the complement of a neighborhood of some
hyperplane to a small neighborhood of $[v_1(x)]$. As $\nu_y$ gives a small
mass to the neighborhood of the hyperplane (uniformly in $y$), it follows
that $(M^n(T^{-n} x)_* \nu_{ T^{-n}  x})$ converges to $\delta_{[v_1(x)]}$.
Thus,
\begin{equation*}
  \mu \otimes \delta_{[v_1]} = \lim (T_{\Pbb}^n)_* (\mu \otimes \nu_{ x}).
\end{equation*}

There is a remarkable consequence of this construction. We can start from a
family of measure $\nu_{x}$ which is invariant under the unstable holonomy
$H^u_{x \to  y}$, i.e., such that $(H^u_{x \to y})_* \nu_{ x} = \nu_{y}$.
Then the same is true of all the iterates $(M^n(T^{-n}
 x)_* \nu_{T^{-n}  x})$. In the limit $n \to
\infty$, it follows that $\delta_{[v_1]}$ is also invariant under unstable
holonomies. (There is something to justify here, as it is not completely
straightforward that the holonomy invariance is invariant under weak
convergence: The simplest way is to work with a one-sided subshift, and then
lift things trivially to the two-sided subshift, see~\cite[Section
4.1]{avila_viana_criterion} for details). This remark leads us to the
following definition.

\begin{definition}
Consider a probability measure $\nu$ on $\Pbb(E)$ which projects to $\mu$
under $\pi$. It is called a $u$-state if, in the fiberwise decomposition
$\nu= \mu \otimes  \nu_{x}$, the measures $\nu_{x}$ are $\mu$-almost surely
invariant under unstable holonomies. It is called an invariant $u$-state if,
additionally, it is invariant under the dynamics.
\end{definition}

The invariant $u$-states can be described under an additional irreducibility
assumption of the cocycle, strong irreducibility.

\begin{definition}
\label{def:strongly_irreducible} We say that a cocycle $M$ with invariant
continuous holonomies over a subshift of finite type is \emph{not strongly
irreducible} if there exist a dimension $0<k<d=\dim E$, an integer $N>0$, and
for each point $x \in \Sigma$ a family of distinct $k$-dimensional vector
subspaces $V_1(x),\dotsc, V_N(x)$ of $E(x)$, depending continuously on $x$,
with the following properties:
\begin{itemize}
\item the family as a whole is invariant under $M$, i.e., for all $x$,
\begin{equation*} M(x)\{V_1(x),\dotsc,V_N(x)\} =
    \{V_1(T x),\dotsc, V_N(T x)\}.
\end{equation*}
\item the family as a whole is invariant under the holonomies, i.e., for
    all $x$ and all $y \in W_{\loc}^u(x)$ one has $H^u_{x \to
    y}\{V_1(x),\dotsc, V_N(x)\} = \{V_1(y),\dotsc, V_N(y)\}$, and the same
    holds for the stable holonomies.
\end{itemize}
Otherwise, we say that $M$ is strongly irreducible.
\end{definition}
In a locally constant cocycle, where holonomies commute (and can therefore be
taken to be the identity), then the holonomy invariance condition reduces to
the condition that each $V_i$ is locally constant, i.e., it only depends on
$x_0$.

The following theorem is the main result of this paragraph. It essentially
follows from the arguments in~\cite{bonatti_viana_lyapunov,
avila_viana_criterion}.

\begin{thm}
\label{thm:u-state-unique} Consider a transitive subshift of finite type $T$
with a Gibbs measure $\mu$. Let $M$ be a locally constant cocycle on a bundle
$E$ over $T$, which is strongly irreducible and has simple top Lyapunov
exponent. Then the corresponding fibered map $T_{\Pbb}$ has a unique
invariant $u$-state, given by $ \mu \otimes \delta_{[v_1]}$ where $v_1(x)$ is
a nonzero vector spanning  the $1$-dimensional Oseledets subspace for the top
Lyapunov exponent at $x$.
\end{thm}

Note that we are assuming that the cocycle is locally constant: This theorem
is wrong if the cocycle only has invariant continuous holonomies, see
Remark~\ref{rmk:not_locally_constant} below.

The rest of this subsection is devoted to the proof of this theorem. We have
already seen that $\mu \otimes \delta_{[v_1]}$ is an invariant $u$-state,
what needs to be shown is the uniqueness. Starting from an arbitrary
$u$-state $\nu$, we have to prove that it is equal to $\mu \otimes
\delta_{[v_1]}$.

As the cocycle is locally constant, one can quotient by the stable direction,
obtaining a unilateral subshift $T_+:\Sigma_+ \to \Sigma_+$ with a Gibbs
measure $\mu_+$, a vector bundle $E_+$ and a cocycle $M_+$. The measure
$\nu^+=(\pi_{E \to E_+})_* \nu$ is then invariant under the fibered dynamics
$T_{+, \Pbb}$. It can be written as $\mu_+ \otimes \nu^+_{x_+}$ for some
measurable family $\nu^+_{x_+}$ of probability measures on $\Pbb(E_+(x_+))$.

The following lemma is~\cite[Proposition 4.4]{avila_viana_criterion}.
\begin{lem}
\label{lem:hatnu_nux} Assume that $\nu$ is an invariant $u$-state. Then the
family of measures $\nu^+_{x_+}$, initially defined for $\mu_+$-almost every
$x_+$, extends to a (unique) family that depends continuously in the weak
topology on all $x_+\in \Sigma_+$.
\end{lem}
For completeness, we sketch the proof, leaving aside the technical details.
\begin{proof}
The measure $\nu^+_{x_+}$ is obtained by averaging all the conditional
measures $\nu_{x}$ over all points $x$ which have the future $x_+$, i.e.,
over the points $(x_-, x_+)$, with respect to a conditional measure
$\dd\mu^-_{x_+}(x_-)$. If $y_+$ is close to $x_+$, one has $y_0=x_0$, so the
possible pasts of $y_+$ are the same as the possible pasts of $x_+$. For any
continuous function $f$ on projective space, we obtain
\begin{equation*}
  \int f \dd\nu^+_{x_+} = \int \left(\int f \dd  \nu_{x_-, x_+}\right) \dd\mu^-_{x_+}(x_-), \quad
  \int f \dd\nu^+_{y_+} = \int \left(\int f \dd  \nu_{x_-, y_+}\right) \dd\mu^-_{y_+}(x_-).
\end{equation*}
When $y_+$ is close to $x_+$, the measures $\dd\mu^-_{x_+}$ and
$\dd\mu^-_{y_+}$ are equivalent, with respective density close to $1$, as we
explained in Paragraph~\ref{subsec:Gibbs}. Moreover, by holonomy invariance
of the conditional measures of $\nu$,
\begin{equation*}
  \int f \dd  \nu_{x_-, y_+} = \int f \circ H^u_{(x_-, x_+) \to (x_-, y_+)} \dd \nu_{x_-, x_+}.
\end{equation*}
By continuity of the holonomies, the function $f \circ H^u_{(x_-, x_+) \to
(x_-, y_+)}$ is close to $f$ if $y_+$ is close to $x_+$. It follows that
$\int f \dd\nu^+_{y_+}$ is close to $\int f \dd\nu^+_{x_+}$, as desired.
Details can be found in~\cite[Section 4.2]{avila_viana_criterion}.
\end{proof}

Henceforth, we write $\nu^+_{x_+}$ for the family of conditional measures,
depending continuously on $x_+$. The next lemma is a version
of~\cite[Proposition 5.1]{avila_viana_criterion} in our setting.

\begin{lem}
\label{lem:hyperplane_0} Assume that $M$ is strongly irreducible in the sense
of Definition~\ref{def:strongly_irreducible}. Let $\nu$ be an invariant
$u$-state, write $\nu^+_{x_+}$ for the continuous fiberwise decomposition of
Lemma~\ref{lem:hatnu_nux}. Then, for any $x_+$, for any hyperplane $L \subset
\Pbb(E_+(x_+))$, one has $\nu^+_{x_+}(L)=0$.
\end{lem}
\begin{proof}
Assume by contradiction that $\nu^+_{x_+}$ gives positive mass to some
hyperplane, for some $x_+$. We will then construct a family of subspaces as
in Definition~\ref{def:strongly_irreducible}, contradicting the strong
irreducibility of the cocycle.

Let $k$ be the minimal dimension of a subspace with positive mass at some
point. Let $\gamma_0$ be the maximal mass of such a $k$-dimensional subspace.
By continuity of $x_+ \mapsto \nu^+_{x_+}$ and compactness, there exist a
point $a_+$ and a $k$-dimensional subspace $V$ with $\nu^+_{a_+}(V) =
\gamma_0$ (\cite[Lemma 5.2]{avila_viana_criterion})

Let $\boV(x_+)$ be the set of all $k$-dimensional subspaces $V$ of $E_+(x_+)$
with $\nu^+_{x_+}(V)=\gamma_0$. Two elements of $\boV(x_+)$ intersect in a
subspace of dimension $<k$, which has measure $0$ by minimality of $k$.
Hence, $\gamma_0 \Card \boV(x_+)= \nu^+_{x_+}(\bigcup_{V \in \boV(x_+)} V)$.
As this is at most $1$, the cardinality of $\boV(x_+)$ is bounded from above,
by $1/\gamma_0$.

Consider a point $b_+$ where the cardinality $N$ of $\boV(b_+)$ is maximal.
For each $V \in \boV(b_+)$, $\nu^+_{b_+}(V)$ is an average of
$\nu^+_{x_+}(M(x_+)^{-1}V)$ over all preimages $x_+$ of $b_+$ under $T_+$
(see~\cite[Corollary 4.7]{avila_viana_criterion}). By maximality, all the
$M(x_+)^{-1}V$ also have mass $\gamma_0$ for $\nu^+_{x_+}$. Iterating this
process, one obtains for all points in $T_+^{-n}\{b_+\}$ at least $N$
subspaces with measure $\gamma_0$ (and in fact exactly $N$ by maximality).
The set $\bigcup_n T_+^{-n}\{b_+\}$ is dense. Hence, any $x_+$ is a limit of
a sequence $x_n$ for which $\boV(x_n)$ is made of $N$ subspaces
$V_1(x_n),\dotsc, V_N(x_n)$. Taking subsequences, we can assume that each
sequence $V_i(x_n)$ converges to a subspace $V_i$, which belongs to $\boV(y)$
by continuity of $y_+\mapsto \nu^+_{y_+}$. Moreover, one has $V_i \neq V_j$
for $i\neq j$: otherwise, the corresponding space would have measure at least
$2 \gamma_0$, contradicting the definition of $\gamma_0$. This shows that the
cardinality of $\boV(x_+)$ is at least $N$, and therefore exactly $N$.

We have shown that the family $\boV(x_+)$ is made of exactly $N$ subspaces
everywhere, that it depends continuously on $x_+$ and that it is invariant
under $T_{+,\Pbb}$. We lift everything to the bilateral subshift $\Sigma$,
setting $\boV(x) = \boV(\pi_{\Sigma \to \Sigma_+} x)$. By construction, the
family is invariant under the dynamics $T_{\Pbb}$. As $V_i$ does not depend
on the past of the points, it is invariant under the stable holonomy (which
is just the identity when one moves along stable sets, thanks to our choice
of trivialization of the bundle).

The family $\boV(x)$ only depends on $x_+$. We claim that, in fact, it only
depends on $x_0$, i.e., it is also invariant under the unstable holonomy. Fix
some $x_+$, and some $y_+$ with $y_0=x_0$. Then $\gamma_0 =
\nu^+_{x_+}(V_i(x_+))$ is an average of the quantities $\nu_{(x_-, x_+)}
(V_i(x_+))$ over all possible pasts $x_-$ of $x_+$. One deduces from this
that $\nu_{(x_-, x_+)} (V_i(x_+)) = \gamma_0$ for almost every such $x_-$,
see~\cite[Lemma 5.4]{avila_viana_criterion}. As $\nu$ is invariant under
unstable holonomy, we obtain $\nu_{(x_-, y_+)} (V_i(x_+)) = \gamma_0$ for
almost every $x_-$. Integrating over $x_-$, we get $\nu^+_{y_+}(V_i(x_+)) =
\gamma_0$. Hence, $V_i(x_+) \in \boV(y_+)$. This shows that
$\boV(x_+)=\boV(y_+)$ if $x_0=y_0$ (almost everywhere and then everywhere by
continuity). Hence, $\boV$ is locally constant. This shows that $M$ is not
strongly irreducible.
\end{proof}
Let us explain how this proof fails if the cocycle are not locally constant,
i.e., if the holonomies do not commute. Let us argue in a trivialization were
the stable holonomies are the identity. The failure is at the end of the
proof, when we show that the family $\boV(x)$ is invariant under unstable
holonomy. We can indeed prove that $\nu_{(x_-, x_+)}(V_i(x_+))= \gamma_0$ for
almost every $x_-$. Then, it follows that $\nu_{(x_-, y_+)} (H^u_{(x_-, x_+)
\to (x_-, y_+)} V_i(x_+)) = \gamma_0$. The problem is that the subspaces
$H^u_{(x_-, x_+) \to (x_-, y_+)} V_i(x_+)$ vary with $x_-$, so one can not
integrate this equality with respect to $x_-$, to obtain a subspace $V$ with
$\nu_{y_+}^+ (V) = \gamma_0$.

\begin{proof}[Proof of Theorem~\ref{thm:u-state-unique}]
Let $\nu$ be a $u$-state, let $\mu\otimes \nu_x$ be its fiberwise
disintegration, and $\nu^+_{x_+}$ the conditional expectation of $\nu_x$ with
respect to the future sigma-algebra. The martingale convergence theorem shows
that, almost surely,
\begin{equation}
\label{eq:nux_limit}
  \nu_x = \lim M^n(T^{-n} x)_* \nu^+_{(T^{-n} x)_+},
\end{equation}
see~\cite[Proposition 3.1]{avila_viana_criterion}.

Let $\epsilon>0$. We may find $\delta$ such that, for any $x_+$ and any
hyperplane $L \subseteq E_+(x_+)$, the $\delta$-neighborhood of $L$ in
$\Pbb(E_+(x))$ (for some fixed distance on projective space) satisfies
$\nu^+_{x_+}(\boN_\delta(L)) \leq \epsilon$, thanks to
Lemma~\ref{lem:hyperplane_0} and continuity of the measures.

Let $E_1(x)=\R v_1(x)$ be the top Oseledets subspace of $M$, and $E_2(x)$ be
the sum of the other subspaces. Let $A$ be a compact subset of $\Sigma$ with
positive measure on which the decomposition $E(x)=E_1(x)\oplus E_2(x)$ is
continuous and on which the convergence in Oseledets theorem is uniform. Fix
$x \in A$. By Poincaré's recurrence theorem, there exists almost surely an
arbitrarily large $n$ such that $T^{-n} x \in A$. In the decomposition $E=E_1
\oplus E_2$, the cocycle $M^n(T^{-n} x)$ is block diagonal, with the first
(one-dimensional) block dominating exponentially the other one. Hence, it
sends $\Pbb(E(T^{-n} x)) \setminus \boN_\delta(E_2(T^{-n} x))$ (whose
$\nu^+_{(T^{-n} x)_+}$-measure is at least $1-\epsilon$ thanks to the choice
of $\delta$) to an $\epsilon$-neighborhood of $E_1(x)$ if $n$ is large
enough. Therefore, $M^n(T^{-n} x)_* \nu^+_{(T^{-n} x)_+}
(\boN_\epsilon([v_1(x)])) \geq 1-\epsilon$. Letting $\epsilon$ tend to $0$,
we get $\nu_x([v_1(x)]) = 1$ thanks to~\eqref{eq:nux_limit}. As the measure
of $A$ can be taken arbitrarily close to $1$, we finally get that $\nu_x$ is
almost everywhere equal to $\delta_{[v_1(x)]}$.
\end{proof}

\begin{rmk}
\label{rmk:not_locally_constant} Theorem~\ref{thm:u-state-unique} is wrong in
general for cocycles which are not locally constant. The difficulty is in
Lemma~\ref{lem:hyperplane_0}: If the cocycle $M$ merely admits invariant
continuous holonomies, there is no reason why the invariant family of
subspaces $\boV(x)$ we construct there should be invariant under the unstable
holonomy, even though $\nu_x$ is. Here is an example of a strongly
irreducible cocycle with simple Lyapunov exponents, over the full shift on
two symbols endowed with any Gibbs measure, which admits two $u$-states.

Let $\Sigma$ be the full shift, let $E=\Sigma \times \R^3$ and let $M$ be the
constant cocycle given by the matrix
$\left(\begin{smallmatrix} 3 & 0 & 0 \\ 0 & 2 & 0 \\
0 & 0 & 1\end{smallmatrix}\right)$. We introduce the holonomies
\begin{align*}
  H_{x \to y}^u & = \left(\begin{matrix}
    1 & 0 & \sum_{n\geq 0} 3^{-n} (y_n - x_n) \\
    0 & 1 & \sum_{n\geq 0} 2^{-n} (y_n - x_n) \\
    0 & 0 & 1
  \end{matrix} \right), \\
\intertext{and}
  H_{x \to y}^s & = \left(\begin{matrix}
    1 & 0 & 0  \\
    0 & 1 & 0 \\
    \sum_{n\leq 0} 3^{n} (y_n - x_n) & \sum_{n\leq 0} 2^{n} (y_n - x_n) & 1
  \end{matrix} \right).
\end{align*}
One checks easily that they are indeed holonomies, and that they are
invariant under $T$. Let $e_i$ denote the $i$-th vector of the canonical
basis. As $e_1$ and $e_2$ are invariant under the unstable holonomies, they
give rise to two distinct $u$-states.

We claim that the cocycle is strongly irreducible. Indeed, consider a nonzero
subbundle $F$ of $E$ which is invariant under $T$ and the holonomies, we will
show that $F=E$. Considering the Oseledets decomposition of $F$ under the
cocycle, it follows that $F$ is spanned by some subfamily $(e_i)_{i\in I}$.
If $1 \in I$ or $2 \in I$, then the invariance under stable holonomy implies
that $3\in I$, since $H^s e_1$ and $H^s e_2$ have a nonzero component along
$e_3$. Hence, $e_3 \in F$ almost everywhere. Then, using the invariance under
unstable holonomy, we deduce that $e_1 \in F$ and $e_2 \in F$ almost
everywhere, as $H^u e_3$ has nonzero components along $e_1$ and $e_2$.
Finally, $F=E$.
\end{rmk}

\subsection{The case of locally constant cocycles}
\label{subsec:locally_constant}

In this paragraph, we prove Theorem~\ref{thm:large_deviations}~(4): if a
cocycle is locally constant above a transitive subshift of finite type, then
it has exponential large deviations for all exponents. The main step is the
following result:

\begin{thm}
\label{thm:strongly_irreducible} Consider a continuous cocycle over a
transitive subshift of finite type endowed with a Gibbs measure, admitting
invariant continuous holonomies. Assume that it has a unique $u$-state. Then
the cocycle has exponential large deviations for its top exponent.
\end{thm}

Before proving this theorem, let us show by successive reductions how it
implies Theorem~\ref{thm:large_deviations}~(4).

\begin{lem}
\label{lem:strongly_irreducible1} Consider a locally constant cocycle which
is strongly irreducible and has simple top Lyapunov exponent, above a
subshift of finite type with a Gibbs measure. Then it has exponential large
deviations for its top exponent.
\end{lem}
\begin{proof}
By Theorem~\ref{thm:u-state-unique}, the cocycle admits a unique $u$-state.
Hence, the result follows from Theorem~\ref{thm:strongly_irreducible}.
\end{proof}

\begin{lem}
\label{lem:strongly_irreducible2} Consider a locally constant cocycle which
has simple top Lyapunov exponent, above a subshift of finite type with a
Gibbs measure. Then it has exponential large deviations for its top exponent.
\end{lem}
\begin{proof}
We argue by induction on dimension of the fibers of the cocycle. Consider a
cocycle $M$ on a bundle $E$ over a subshift of finite type $T$, with simple
top Lyapunov exponent. We will show that it has exponential large deviations
for its top exponent, assuming the same results for all cocycles on fiber
bundles with strictly smaller dimension. We will prove the lower
bound~\eqref{eq:lower_bound} (with $i=1$) for $M$.

If the cocycle $M$ is strongly irreducible, then the result follows from
Lemma~\ref{lem:strongly_irreducible1}, so assume that it is not. Consider a
locally constant invariant family $V_1(x),\dotsc, V_N(x)$ as in
Definition~\ref{def:strongly_irreducible}, such that $N$ is minimal. Let
$V(x)$ be the span of $V_1(x),\dotsc, V_N(x)$. It is also locally constant
and invariant under the cocycle and the holonomies.

Assume first that the dimension of $V$ is strictly smaller than that of $E$.
Define a cocycle $M_{V}$ as the restriction of $M$ to $V$, and a cocycle
$M_{E/ V}$ as the cocycle induced by $M$ on the quotient bundle $E/ V$. These
two cocycles are locally constant. By definition of the restriction norm and
the quotient norm, one has
\begin{equation}
\label{eq:quotient_restrict}
  \norm{M^n(x)} \geq \max (\norm{M_{V}^n(x)},
    \norm{M_{E/ V}^n(x)}).
\end{equation}
Moreover, by Lemma~\ref{lem:flag_split}, one of the two cocycles has
$\lambda_1$ as a simple top Lyapunov exponent, and these two cocycles are
locally constant and have strictly smaller fiber dimension. By our induction
assumption, we deduce that
\begin{equation*}
  \mu \{x \st \log \norm{M_W^n(x)}
    \leq n \lambda_1 - n \epsilon\} \leq C e^{-C^{-1}n},
\end{equation*}
where $W$ is either $V$ or $E/V$. The same bound follows for $M$ thanks
to~\eqref{eq:quotient_restrict}.

\smallskip

Assume now that the dimension of $V$ is equal to that of $E$, i.e., $V=E$.
Consider a new dynamics $\tilde T$, on $\tilde \Sigma= \Sigma \times
\{1,\dotsc, N\}$, mapping $(x, i)$ to $(Tx, j)$ where $j=j(x, i)$ is the
unique index such that $M(x) V_i(x) = V_j(T x)$. As $M$ and all the $V_k$
only depend on $x_0$, the function $j$ only depends on $i$, $x_0$ and $x_1$.
Hence, $\tilde T$ is a subshift of finite type. As we chose $N$ to be
minimal, there is no invariant proper subfamily of $V_1,\dotsc, V_N$. Hence,
$\tilde T$ is a transitive subshift. Let also $\tilde \mu$ be the product
measure of $\mu$ and the uniform measure on $\{1,\dotsc,N\}$, it is again a
Gibbs measure for $\tilde T$, therefore ergodic by transitivity.

Above $\tilde \Sigma$, we consider a new bundle $\tilde E(x, i) = V_i(x)$,
and the resulting cocycle $\tilde M$ which is the restriction of $M$ to
$V_i$. On any $E(x)$, one can find a basis made of vectors in the subspaces
$V_i(x)$, by assumption. It follows that $\norm{M^n(x)} \leq C
\max_i\norm{M^n(x)_{|V_i(x)}}$, for some uniform constant $C$. Hence, the top
Lyapunov exponent of $\tilde M$ is (at least, and therefore exactly)
$\lambda_1$. Moreover, it is simple as the top Oseledets space for $\tilde M$
in $\tilde E(x, i)$ is included in the top Oseledets space for $M$ in $E(x)$,
which is one-dimensional by assumption.

By our induction assumption, we obtain the bound~\eqref{eq:lower_bound} with
$i=1$ for the cocycle $\tilde M$ over the subshift $\tilde T$ and the measure
$\tilde \mu$ (note that it is important there that we have formulated the
induction assumption for all subshifts of finite type, not only the original
one). The result follows for the original cocycle as $\norm{M^n(x)} \geq
\norm{\tilde M^n(x, 1)}$ for all $x$.
\end{proof}

\begin{lem}
\label{lem:strongly_irreducible3} Consider a locally constant cocycle, above
a subshift of finite type with a Gibbs measure. Then it has exponential large
deviations for its top exponent.
\end{lem}
\begin{proof}
Consider a locally constant cocycle $M$ for which the multiplicity $d$ of the
top Lyapunov exponent is $>1$. Then the top Lyapunov exponent of $\Lambda^d
M$ is simple, equal to $d\lambda_1$. Moreover, for any matrix $A$ (with
singular values $s_1 \geq s_2 \geq \dotsc$), we have $\norm{A}^d =s_1^d \geq
\norm{\Lambda^d A}=s_1 \dotsm s_d$. Thus,
\begin{equation*}
  \{x \st \log \norm{M^n(x)} \leq n \lambda_1(M) - n\epsilon \}
  \subseteq \{ x \st \log \norm{\Lambda^d  M^n(x)}
    \leq n \lambda_1(\Lambda^d  M) - nd\epsilon\}.
\end{equation*}
The last set has an exponentially small measure by
Lemma~\ref{lem:strongly_irreducible2}, as $\Lambda^d M$ has a simple top
Lyapunov exponent by construction, and is locally constant. The desired bound
follows for $M$.
\end{proof}

\begin{proof}[Proof of Theorem~\ref{thm:large_deviations}~(4)]
\label{proof:4} Proving exponential large deviations for the cocycle $M$ and
some index $i$ amounts to proving exponential large deviations for $\Lambda^i
M$ and its top Lyapunov exponent. Hence, the theorem follows from
Lemma~\ref{lem:strongly_irreducible3}.
\end{proof}

The rest of this paragraph is devoted to the proof of
Theorem~\ref{thm:strongly_irreducible}. The proof follows the classical
strategy of Guivarc'h Le Page for products of independent matrices (with the
uniqueness of the $u$-state replacing the uniqueness of the stationary
measure), although the technical details of the implementation are closer for
instance to~\cite[Proof of Theorem~1]{dolgopyat_limit}.

Henceforth, we fix a transitive subshift of finite type $T:\Sigma \to \Sigma$
with a Gibbs measure $\mu$, and a continuous cocycle $M:E \to E$ above $T$
which admits a unique $u$-state denoted by $\nu_u$. Changing coordinates in
$E$ using the unstable holonomy, we can assume without loss of generality
that $M(x)$ only depends on the past $x_-$ of $x$.

We denote by $\Sigma_-$ the set of pasts of points in $\Sigma$. The left
shift $T$ does not induce a map on $\Sigma_-$ (it would be multivalued, since
there would be a choice for the zeroth coordinate), but the right shift
$T^{-1}$ does induce a map $U$ on $\Sigma-$. This is a subshift of finite
type, for which the induced measure $\mu_- = (\pi_{\Sigma \to \Sigma_-})_*
\mu$ is invariant (and a Gibbs measure).

The measure $\mu$ has conditional expectations $\mu_{x_-}^+$ above its factor
$\mu_-$: it can be written as $\mu = \mu_- \otimes \mu_{x_-}^+$. The family
of measures $\mu_{x_-}^+$ is canonically defined for all point $x_- \in
\Sigma_-$, and varies continuously with $x_-$, as we explained in
Paragraph~\ref{subsec:Gibbs} (for the opposite time direction).

To any point $(x,[v]) \in \Pbb(E)$, we associate a measure $\nu_{(x,[v])}$ on
$\Pbb(E)$ as follows. There is a canonical lift to $E$ of $W_{\loc}^u(x_-)$,
going through $v$, given by $\{(x_-, y_+, H^u_{(x_-,x_+) \to (x_-, y_+)} v
\}$. The measure $\mu_{x_-}^+$ on $W_{\loc}^u(x_-)$ can be lifted to this
set, giving rise after projectivization to the measure $\nu_{(x,[v])}$. This
measure is invariant under (projectivized) unstable holonomy, it projects to
$\mu_{x_-}^+$ under the canonical projection $\Pbb(E) \to \Sigma$, and it
projects to $\delta_{x_-}$ under the canonical projection $\Pbb(E) \to
\Sigma_-$. By construction, for any $x_-$, $x_+$ and $y_+$,
\begin{equation*}
  \nu_{(x_-, x_+, [v])} = \nu_{(x_-, y_+, [H^u_{(x_-, x_+) \to (y_-, y_+)} v])}.
\end{equation*}

More generally, finite averages or even integrals of such measures are again
$H^u$-invariant.

\medskip

There is a natural Markov chain on $\Sigma_-$, defined as follows. A point
$x_-$ has several preimages $y_-^i$ under $U$. By the invariance of the
measure $\mu_-$, the sum $1/J(y_-^i)$ is equal to $1$, where $J$ is the
jacobian of $U$ for $\mu_-$. Hence, one defines a Markov chain, by deciding
to jump from $x_-$ to $y_-^i$ with probability $1/J(y_i^i)$. The
corresponding Markov operator is given by
\begin{equation*}
  \boL v(x_-) = \sum_{U(y_-) = x_-} \frac{1}{J(y_-)} v (y_-).
\end{equation*}
This is simply the transfer operator of Paragraph~\ref{subsec:Gibbs} (for the
map $U$ instead of the map $T$). Replacing the potential $\phi$ which defines
the Gibbs measure by a cohomologous potential, we may write $\frac{1}{J(y_-)}
= e^{\phi(y)}$.

Correspondingly, we define an operator $\boM$ acting on measures on
$\Pbb(E)$, by
\begin{equation*}
  \boM \nu = (T_\Pbb)_* \nu.
\end{equation*}
It maps $\nu_{(x,[v])}$ (supported on the lift $V$ of $W_{\loc}^u(x_-)$
through $[v]$) to a measure supported on $T_\Pbb V$ (which is a lift of the
union of the unstable manifolds $W_{\loc}^u(y_-)$ for $U(y_-) = x_-$). Choose
on each of these submanifolds a point $(y_-, y_+, [v_y])$ (where $y_+$ is
arbitrary, and $[v_y]$ is the unique vector in $TV$ above $(y_-, y_+)$). Then
we have
\begin{equation}
\label{eq:iterate_boM}
  \boM \nu_{(x,[v])} = \sum e^{\phi(y_-)} \nu_{(y_-, y_+, [v_y])}.
\end{equation}
This follows from the equation~\eqref{eq:T-inv} for the evolution of the
conditional measures under the dynamics, and then the uniqueness of the
$H^u$-invariant lift.

\begin{prop}
\label{prop:limboMn} Let $f$ be a continuous function on $\Pbb(E)$. Then,
uniformly in $x \in \Sigma$ and $v \in \Pbb(E)$, when $N \to \infty$,
\begin{equation*}
   \frac{1}{N} \sum_{n=0}^{N-1}\int f \dd \boM^n \nu_{(x,[v])} \to \int f \dd \nu_u,
\end{equation*}
where $\nu_u$ is the unique invariant $u$-state of $M$.
\end{prop}
\begin{proof}
It suffices to show that any weak limit $\nu_\infty$ of sequences of the form
\begin{equation*}
  \nu_N= \frac{1}{N} \sum_{n=0}^{N-1}\boM^n \nu_{(x_N,[v_N])}
\end{equation*}
(where $x_N$ and $[v_N]$ may vary with $N$) is an invariant $u$-state.

The invariance of the limiting measure is clear from the Cesaro-averaging and
the definition $\boM \nu = (T_\Pbb)_* \nu$. The $H^u$ invariance also follows
from the construction. It remains to show that $\nu_\infty$ projects to $\mu$
on $\Sigma$ or, equivalently, that it projects to $\mu_-$ on $\Sigma_-$.

The projection of $\nu_N$ on $\Sigma_-$ is the Cesaro average of $\sum_{U^n
y_- = (x_N)_-} e^{S_{-n} \phi(y_-)} \delta_{y_-}$, i.e., the position at time
$n$ of the Markov chain started from $x_N$ at time $0$. For any continuous
function $v$ on $\Sigma_-$, we get $\int v \dd\pi_* \nu_N = \frac{1}{N}
\sum_{n=0}^{N-1} \boL^n v((x_N)_-)$. By a classical property of transfer
operators (see~\eqref{eq:conv_boL}), this converges uniformly to $\int v
\dd\mu_-$. This proves that the only possible weak limit for $\pi_*(\nu_N)$
is $\mu_-$.
\end{proof}

Fix once and for all $\epsilon>0$, for which we want to prove the inequality
\begin{equation}
\label{eq:mu_to_prove}
  \mu\{x \st \log \norm{M^n(x)} \leq n(\lambda_1 - \epsilon)\} \leq C e^{-C^{-1}n}.
\end{equation}

\begin{lem}
\label{lem:g0} Define a function $g_0$ on  $\Pbb(E)$ by
\begin{equation*}
  g_0([x], v) = \log (\lambda_1 - \epsilon) - \log (\norm{M(x) v}/\norm{v} ),
\end{equation*}
where the last term in this formula does not depend on the choice of the lift
$v$ of $[v]$. Then there exist $N$ and $\alpha,\beta>0$ such that, for any
$x$ and $v$,
\begin{equation*}
  \int e^{\alpha S_N g_0} \dd \nu_{(x,[v])} \leq e^{-\beta N}.
\end{equation*}
\end{lem}
\begin{proof}
Define a function $f_0$ on $\Pbb(E)$ by
\begin{equation*}
  f_0(x, [v]) = \log (\norm{M(x) v}/\norm{v} ).
\end{equation*}
The integral of $f_0$ with respect to the unique invariant $u$-state $\nu_u$
measures the average expansion of a vector in the maximally expanded
Oseledets subspace, which is by definition equal to the maximal Lyapunov
exponent $\lambda_1$. Hence, it is not difficult to check the following
formula, due to Furstenberg (see for instance~\cite[Proposition
6.5]{viana_lyapunov}):
\begin{equation*}
  \int f_0 \dd \nu_u = \log \lambda_1.
\end{equation*}
It follows that
\begin{equation*}
  \int g_0 \dd \nu_u = \log (\lambda_1 - \epsilon) - \log \lambda_1 < 0.
\end{equation*}
Fix some $c_0>0$ such that $\int g_0 \dd \nu_u < -c_0$. By
Proposition~\ref{prop:limboMn}, there exists an integer $N$ such that, for
any $x$ and $v$,
\begin{equation*}
  \frac{1}{N} \sum_{n=0}^{N-1}\int g_0 \dd \boM^n \nu_{(x,[v])} \leq -c_0.
\end{equation*}
By definition of $\boM$, we get
\begin{equation*}
  \int S_N g_0 \dd \nu_{(x,[v])} = \sum_{n=0}^{N-1}\int g_0 \dd \boM^n \nu_{(x,[v])} \leq -c_0 N.
\end{equation*}

Using the inequality $e^t \leq 1+t+t^2 e^{\abs{t}}$, we obtain for any
$\alpha \in (0,1)$
\begin{equation*}
  \int e^{\alpha S_N g_0} \dd \nu_{(x,[v])} \leq 1+\alpha \int S_N g_0 \dd \nu_{(x,[v])}
  + \alpha^2 \int (S_N g_0)^2 e^{\abs{S_N g_0}} \dd \nu_{(x,[v])}
  \leq 1 -\alpha c_0 N + \alpha^2 C,
\end{equation*}
where $C$ is a constant depending on $N$ but not on $\alpha$. (For the bound
in the last term, note that the function $S_N g_0$ is uniformly bounded, as a
continuous function on a compact space.) When $\alpha$ is small enough, the
term $\alpha^2 C$ is negligible. Hence, we obtain for small enough $\alpha$
and for $\beta = \alpha c_0/2$ the inequality
\begin{equation*}
  \int e^{\alpha S_N g_0} \dd \nu_{(x,[v])} \leq 1 - \beta N \leq e^{-\beta N}.
\qedhere
\end{equation*}
\end{proof}

\begin{lem}
\label{lem:exp_control} There exists a constant $C$ such that, for any $n \in
\N$ and any $x$ and $v$, one has
\begin{equation*}
  \int e^{\alpha S_n g_0} \dd \nu_{(x,[v])} \leq C e^{-\beta n}.
\end{equation*}
\end{lem}
\begin{proof}
It suffices to prove the lemma for times of the form $nN$, as the general
case only results in an additional multiplicative constant.

Fix some $n$. Iterating~\eqref{eq:iterate_boM}, $(T_{\Pbb}^{nN})_*
\nu_{(x,[v])}$ is a finite linear combination of measures of the form
$\nu_{(x_i, [v_i])}$, with some coefficients $c_i >0$ adding up to $1$. Then
\begin{equation*}
  \int e^{\alpha S_{(n+1)N} g_0} \dd \nu_{(x,[v])}
  = \sum_i c_i \int e^{\alpha S_{nN} g_0 \circ T_{\Pbb}^{-nN}} \cdot e^{\alpha S_N g_0} \dd\nu_{(x_i, [v_i])}.
\end{equation*}
In each of the integrals, the term $e^{\alpha S_{nN} g_0 \circ T^{-nN}}$ is
constant as $g_0$ and $M$ only depend on the past of points in $\Sigma$.
Hence, this integral is a constant multiple of $\int e^{\alpha S_N g_0}
\dd\nu_{(x_i, [v_i])}$, which is $\leq e^{-\beta N}$ by Lemma~\ref{lem:g0}.
We get
\begin{equation*}
  \int e^{\alpha S_{(n+1)N} g_0} \dd \nu_{(x,[v])}
  \leq e^{-\beta N} \sum_i c_i \int e^{\alpha S_{nN} g_0 \circ T^{-nN}}\dd\nu_{(x_i, [v_i])}
  = e^{-\beta N} \int e^{\alpha S_{nN} g_0} \dd \nu_{(x,[v])}.
\end{equation*}
The conclusion then follows by induction on $n$.
\end{proof}

\begin{proof}[Proof of Theorem~\ref{thm:strongly_irreducible}]
Fix some vector $v$. Then the average $\int \nu_{(x,[v])} \dd\mu(x)$ is a
measure on $\Pbb(E)$ that projects to $\mu$. If $\log \norm{M^n(y)} \leq n
\lambda_1(M) - n\epsilon$, then for any vector $w$ one has $\log(\norm{M^n(y)
w} / \norm{w}) \leq n(\lambda_1(M) - \epsilon)$, i.e., $S_n g_0(y, w) \geq
0$. We obtain
\begin{align*}
  \mu\{y \st \log \norm{M^n(y)} \leq n \lambda_1(M) - n\epsilon\}
  &\leq \int 1(S_n g_0(y,w) \geq 0) \dd \nu_{(x,[v])}(y,w) \dd\mu(x)
  \\&\leq \int \left( \int e^{\alpha S_n  g_0} \dd \nu_{(x,[v])}\right) \dd\mu(x).
\end{align*}
By Lemma~\ref{lem:exp_control}, the last integral is bounded by $C e^{-\beta
n}$. The upper bound~\eqref{eq:mu_to_prove} follows.
\end{proof}

\subsection{Proof of Theorem~\ref{thm:large_deviations}~(5)}
\label{proof:5}

Consider a cocycle $M$ admitting invariant continuous holonomies, which is
pinching and twisting in the sense of Avila-Viana. We want to show that it
admits exponential large deviations for all exponents.

\cite{avila_viana_criterion} shows that there is a unique invariant $u$-state
on $\Pbb(E)$, corresponding to the maximally expanded Oseledets subspace, see
the first lines of Section~7 in~\cite{avila_viana_criterion}. Hence,
Theorem~\ref{thm:strongly_irreducible} applies and shows that $M$ has
exponential large deviations for its top exponent.

To prove exponential large deviations for an exponent $i$, a natural strategy
would be to consider the cocycle $\Lambda^i M$ and prove that it has
exponential large deviations for its top Lyapunov exponent. However, there is
no reason why $\Lambda^i M$ should be twisting and pinching. What Avila and
Viana prove in~\cite{avila_viana_criterion}, however, is that $M$ has a
unique $u$-state on the Grassmannian of $i$-dimensional subspaces. All the
arguments in the proof of Theorem~\ref{thm:strongly_irreducible} go through
if one replaces everywhere the space $\Pbb(E)$ by the corresponding
Grassmannian. Then the Grassmannian version of
Theorem~\ref{thm:strongly_irreducible} shows that $M$ has exponential large
deviations for the exponent $i$.%
\qed

\subsection{Proof of Theorem~\ref{thm:large_deviations}~(6)}
\label{proof:6}

Consider a two dimensional cocycle $M$ admitting continuous holonomies, we
want to show that it satisfies exponential large deviations for all exponents
$i$. For $i=2$, the norm $\norm{\Lambda^i M(x)}$ is the absolute value of the
determinant of $M(x)$. The desired estimate~\eqref{eq:exp_dev_all_exp}
involves an additive cocycle, the Birkhoff sums of the continuous function
$\log \abs{\det M(x)}$. Hence,~\eqref{eq:exp_dev_all_exp} follows from the
large deviations estimate for Birkhoff sums.

The only non-trivial case is $i=1$, i.e., exponential large deviations for
$\norm{M^n(x)}$. If $M$ admits a unique invariant $u$-state on $\Pbb(E)$,
then the result follows from Theorem~\ref{thm:strongly_irreducible}, and we
are done. If the two Lyapunov exponents of $M$ are equal, then the result
follows from Theorem~\ref{thm:large_deviations}~(1). The last case is when
the Lyapunov exponents are distinct, but there are two different invariant
$u$-states. The non-uniqueness implies that something fails if we try to
follow the proof of Theorem~\ref{thm:u-state-unique}. The only place in the
proof of this theorem where we used the fact that the cocycle is locally
constant is in the proof of Lemma~\ref{lem:hyperplane_0}. Without this
assumption, the proof of this lemma constructs a family $\boV(x)$ of
subspaces (which are necessarily one-dimensional), invariant under the
dynamics and the stable holonomy, but not necessarily under the unstable
holonomy. In general, this prevents us from implementing the induction
argument of Lemma~\ref{lem:strongly_irreducible2} as the induced cocycle on
$\boV$ and the quotient cocycle do not admit invariant continuous holonomies
any more. However, in the specific case of $2$-dimensional cocycles, the
induced cocycles and the quotient cocycles are both $1$-dimensional.
Therefore, they satisfy exponential large deviations thanks to
Theorem~\ref{thm:large_deviations}~(1). Hence, the argument in
Lemma~\ref{lem:strongly_irreducible2} goes through to prove that the original
cocycle also satisfies exponential large deviations.
\qed

\section{Exponential returns to nice sets for subadditive cocycles}

The main statement of this section is the following theorem. Note that the
assumptions of the theorem ensure that the function $F$ below is finite
almost everywhere, although it can be infinite on points which are not
typical for $\mu$. We are trying to control how large it will be along
typical orbits, in a quantitative sense.

\begin{thm}
\label{thm:exp_returns} Let $T:X \to X$ be a continuous map preserving an
ergodic probability measure $\mu$ on a compact space. Consider a subadditive
cocycle $u:\N \times X \to \R$, such that $u(n,x)/n$ converges almost
everywhere to $0$, and $u(n, \cdot)$ is continuous for all $n$. Let also
$\epsilon>0$. Define a function
\begin{equation*}
  F(x) = \sup_{n\geq 0}\, \abs{u(n,x)} - \epsilon n.
\end{equation*}
Assume that $u$ has exponential large deviations, and that the Birkhoff sums
of continuous functions also have exponential large deviations.

Let $\delta>0$. Then there exists $C>0$ such that, for any $n\geq 0$,
\begin{equation*}
  \mu \{x \st \Card\{j \in [0, n-1] \st F(T^j x) > C\} \geq \delta n\}
  \leq C e^{-C^{-1}n}.
\end{equation*}
\end{thm}

In the applications we have in mind, $u$ will be of the form $u(n,x) = \log
\norm{\Lambda^i M^{(n)}(x)}-n(\lambda_1+\dotsb+\lambda_i)$, for some cocycle
$M$ with Lyapunov exponents $\lambda_k$. The points where $F(x) \leq C$ are
the points where all the iterates of the cocycle are well controlled.
Essentially, they belong to some Pesin sets (see
Proposition~\ref{thm:deterministic_Oseledets} below for a precise version of
this statement). Hence, the lemma will imply that most iterates of a point
return often to Pesin sets, if the matrix cocycle has exponential large
deviations for all exponents.

The proof is most conveniently written in terms of superadditive cocycles.
Note that, in the lemma below, the definition of $G$ resembles that of $F$ in
the theorem above, except for the lack of absolute value. Hence, the
following lemma applied to $v(n,x) = -u(n,x)-n\epsilon$ proves one of two
inequalities in Theorem~\ref{thm:exp_returns}.

\begin{lem}
\label{lem:exp_returns} Let $T:X \to X$ preserve an ergodic probability
measure $\mu$ on a compact space. Consider a superadditive cocycle $v:\N
\times X \to \R$, such that $v(n,x)/n$ converges almost everywhere to
$-\epsilon<0$, and $v(n, \cdot)$ is continuous for all $n$. Define a function
\begin{equation*}
  G(x) = \sup_{n\geq 0} v(n,x).
\end{equation*}
Assume that $v$ satisfies exponential large deviations, and that the Birkhoff
sums of continuous functions also satisfy exponential large deviations.

Let $\delta>0$. Then there exists $C>0$ such that, for any $n\geq 0$,
\begin{equation*}
  \mu \{x \st \Card\{j \in [0, n-1] \st G(T^j x) > C\} \geq \delta n\}
  \leq C e^{-C^{-1}n}.
\end{equation*}
\end{lem}
\begin{proof}
When $N$ tends to $+\infty$, the sequence $v_N/N$ tends almost surely to
$-\epsilon$. The convergence also holds in $L^1$ by Kingman's theorem. Then
$v_N/N + \epsilon$ tends almost surely and in $L^1$ to $0$. Then $\min(v_N/N
+ \epsilon, 0)$ tends almost surely and in $L^1$ to $0$. Thus, we can take
once and for all a large enough $N$ so that
\begin{equation}
\label{eq:int_w}
  \int \min(v_N/N + \epsilon, 0) \geq - \delta \epsilon/10.
\end{equation}
Let $w = v_N/N$.

By Lemma~\ref{lem:anx_bound_Birkhoff} applied to the subadditive cocycle
$-v$, there exists a constant $C_0>0$ such that $v(n,x) \geq S_n w (x) -
C_0$, for any $x\in X$ and any $n\in \N$. We will show that
\begin{equation*}
  \mu \{x \st \Card\{j \in [0, n-1] \st G(T^j x) > 2C_0\} \geq \delta n\}
\leq C e^{-C^{-1}n}.
\end{equation*}

Assume first that $x$ has an iterate where the cocycle is large along an
extremely long interval, i.e., $x$ belongs to
\begin{equation*}
  K_n = \bigcup_{t=0}^{n-1} (T^t)^{-1} \{y \st \exists j \geq \delta n/2, v(j,y) > 0\}.
\end{equation*}
As $v$ has exponential large deviations and converges to a negative constant,
the last set has a measure which is exponentially small in $n$. As $T$ is
measure-preserving, it follows that $\mu(K_n)$ is also exponentially small.

Consider now $x\notin K_n$ such that $\Card\{j \in [0, n-1] \st G(T^j x)
> 2C_0\} \geq \delta n$. Then
\begin{equation}
\label{eq:decompose_pos}
  \Card\{j \in [0, n-1-\delta n/2] \st G(T^j x) > 2C_0\} \geq \delta n/2.
\end{equation}

We define inductively a sequence of times $t_k$ as follows. We start from
$t_0=0$. If $G(T^{t_k} x) > 2C_0$ and $t_k\leq n-1 -\delta n/2$, then we say
that $t_k$ belongs to the set $U^+$ of sum-increasing times. In this case, we
can choose $n_k>0$ such that $v(n_k, T^{t_k}x) > 2C_0$, by definition of $H$.
Then we let $t_{k+1} = t_k+n_k$. Otherwise, we say that $t_k$ belongs to the
set $U^-$ of sum-decreasing times, and we let $t_{k+1} = t_k+1$. We stop at
the first $t_j$ where $t_j \geq n$.

Let $A^+ = \bigcup_{t_k\in U^+} [t_k, t_{k+1})$, and $A^- = [0,n-1] \setminus
A^+$. As $x \notin K_n$, the lengths $n_k=t_{k+1}-t_k$ when $t_k \in U^+$ are
all bounded by $\delta n/2$. Hence, $A^+$ is included in $[0,n-1]$. Moreover,
the set of bad times, on the left of~\eqref{eq:decompose_pos}, is included in
$A^+$. Therefore, $\Card A^+ \geq \delta n/2$, and $\Card A^- \leq
(1-\delta/2)n$.

We will also need to write the set $A^-$ as a union of intervals $\bigcup
[t'_j, t'_j + n'_j)$ over some index set $J$, i.e., we group together the
times in $U^-$ that are not separated by times in $U^+$.

Using the decomposition of $[0,n-1]$ as $A^+ \cup A^-$, the decomposition of
these sets into intervals, and the superadditivity of the cocycle, we obtain
the inequality
\begin{equation*}
  v(n, x) \geq \sum_{t_k \in U^+} v(n_k, T^{t_k} x) + \sum_{j\in J} v(n'_j, T^{t'_j}x)
  \geq \sum_{t_k \in U^+} 2C_0 + \sum_{j\in J} v(n'_j, T^{t'_j}x),
\end{equation*}
where the last inequality follows from the definition of $U^+$. Note that the
right endpoint of an interval in $A^-$ belongs to $U^+$, except for the last
interval. It follows that $\Card J \leq \Card U^+ + 1 \leq 2 \Card U^+$.
Hence, the above inequality implies
\begin{equation*}
  v(n, x) \geq \sum_{j\in J} (C_0 + v(n'_j, T^{t'_j}x)).
\end{equation*}
Together with the definition of $C_0$, this gives
\begin{equation*}
  v(n, x) \geq \sum_{j\in J} S_{n'_j} w(T^{t'_j}x)
  = \sum_{k\in A^-} w(T^k x).
\end{equation*}
Now, let us introduce $\epsilon$:
\begin{align*}
  v(n, x) & \geq \sum_{k\in A^-} (w(T^k x) + \epsilon) - \epsilon \Card(A^-)
  \\& \geq \sum_{k\in [0,n-1]} \min(w(T^k x) + \epsilon, 0) - \epsilon \Card(A^-)
  \\& \geq \sum_{k\in [0,n-1]} \min(w(T^k x) + \epsilon, 0) - \epsilon (1-\delta/2)n,
\end{align*}
where the last inequality holds as $\Card A^- \leq (1-\delta/2)n$.

The continuous function $x\mapsto \min(w(x) + \epsilon, 0)$ has exponential
large deviations and integral $\geq -\delta\epsilon/10$ by ~\eqref{eq:int_w}.
Hence, we have $\sum_{k\in [0,n-1]} \min(w(T^k x) + \rho, 0) \geq - n
\delta\epsilon/5$ apart from an exponentially small set. Apart from this set,
we obtain
\begin{equation*}
  v(n,x) \geq -\epsilon n + (\delta/2-\delta/5)\epsilon n.
\end{equation*}
As $v$ has exponential large deviations and asymptotic average $-\epsilon$,
it follows that this condition on $x$ has exponentially small measure.
\end{proof}

\begin{proof}[Proof of Theorem~\ref{thm:exp_returns}]
The function $F$ is the maximum of the two functions
\begin{equation*}
  H(x) = \sup_{n\geq 0} -u(n,x) - n\epsilon,\quad I(x) = \sup_{n\geq 0} u(n,x) - n\epsilon.
\end{equation*}
We should show that each of these functions satisfies the conclusion of the
theorem. For $H$, this follows from Lemma~\ref{lem:exp_returns} applied to
$v(n,x) = -u(n,x) -n\epsilon$.

For $I$, let us consider $N>0$ such that $u_N/N$ has integral $<\epsilon/2$.
By Lemma~\ref{lem:anx_bound_Birkhoff}, there exists a constant $C_0$ such
that $u(n,x) \leq S_n(u_N/N) + C_0$ for all $n$. Let $w = u_N/N-\epsilon$.
Lemma~\ref{lem:exp_returns} applied to the cocycle $S_n w$ shows that, for
some constant $C_1>0$,
\begin{equation*}
  \mu \{x \st \Card\{j \in [0, n-1] \st \sup_n S_n w(T^j x) > C_1\} \geq \delta n\}
  \leq C e^{-C^{-1}n}.
\end{equation*}
If $u(n,x)-n\epsilon \geq C_0+C_1$, then $S_n w(x) \geq C_1$. Hence, the
control on $I$ follows from the previous equation.
\end{proof}

\section{A deterministic control on the Pesin function}

An important difficulty to prove Theorem~\ref{thm:exp_returns_Pesin} is that
the Pesin function $A_\epsilon$ is defined in terms of the Oseledets
subspaces $E_i(x)$, which vary only measurably with the point and for which
we have no good control. On the other hand, Theorem~\ref{thm:exp_returns}
provides exponentially many returns for sets defined in terms of functions
for which we have good controls, e.g., Birkhoff sums of continuous functions
(by the large deviation principle) or norms of linear cocycles (if one can
prove exponential large deviations for them, using for instance
Theorem~\ref{thm:large_deviations}). Our goal in this section is to explain
how controls on such quantities imply controls on the Pesin function
$A_\epsilon$. Then, Theorem~\ref{thm:exp_returns_Pesin} will essentially
follow from Theorem~\ref{thm:exp_returns}. To prove such a result, we need to
revisit the proof of Oseledets theorem and replace almost sure controls with
more explicit bounds.

Consider an invertible map $T:X\to X$ preserving a probability measure $\mu$,
and a log-integrable linear cocycle $M$ above $T$ on $X \times \R^d$. Let
$\lambda_1 \geq \dotsb \geq \lambda_d$ be its Lyapunov exponents, let $I=\{i
\st \lambda_i<\lambda_{i-1}\}$ be a set of indices for the distinct Lyapunov
exponents, let $E_i$ be the Lyapunov subspaces.

Given $\epsilon>0$, define functions
\begin{align}
  \notag
  B^+_\epsilon(x) & = \sup_{i\in [1,d]} B^{(i)+}_\epsilon = \sup_{i\in [1,d]} \sup_{n\geq 0}\,
                  \abs{ \log \norm{\Lambda^i M^{(n)}(x)} - n(\lambda_1+\dotsb+\lambda_i)} - n\epsilon,
  \\ \notag
  B^-_\epsilon(x) & = \sup_{i\in [1,d]} B^{(i)-}_\epsilon = \sup_{i\in [1,d]} \sup_{n\leq 0}\,
                  \abs{ \log \norm{\Lambda^i M^{(n)}(x)} - n (\lambda_d+\dotsb+\lambda_{d-i+1})} - \abs{n}\epsilon
\intertext{and}
  \label{eq:def_B_epsilon}
  B_\epsilon(x) & = \max(B^+_\epsilon(x), B^-_\epsilon(x)).
\end{align}
These are the functions we can control using the tools of the previous
sections.

The following proposition asserts that a control on $B_\epsilon$ and a mild
control on angles implies a control on $A_{\epsilon'}$ for $\epsilon' =
20d\epsilon$. For $i\in I$, let us denote by $F^{(m)}_{\geq i}(x)$ the
maximally contracted subspace of $M^{(m)}(x)$ of dimension $d-i+1$, and by
$F^{(-m)}_{<i}(x)$ the maximally contracted subspace of $M^{(-m)}(x)$ of
dimension $i$, if these spaces are uniquely defined, as in the statement of
Theorem~\ref{thm:Oseledets_limit}.

\begin{thm}
\label{thm:deterministic_Oseledets} Assume that $\norm{M(x)}$ and
$\norm{M(x)^{-1}}$ are bounded uniformly in $x$. Consider $\epsilon \in (0,
\min_{i\neq j \in I} \abs{\lambda_i-\lambda_j}/(20d))$ and $\rho>0$ and
$C>0$. Then there exist $m_0\in \N$ and $D>0$ with the following properties.

Consider a point $x$ satisfying $B_{\epsilon}(x) \leq C$. Then its subspaces
$F^{(n)}_{\geq i}(x)$ and $F^{(-n)}_{\leq i}(x)$ are well defined for all $n
\geq m_0$, and converge to subspaces $F^{(\infty)}_{\geq i}(x)$ and
$F^{(-\infty)}_{\leq i}(x)$.

Assume additionally that, for all $i\in I$, there exists $m \geq m_0$ such
that the angle between $F^{(m)}_{\geq i}(x)$ and $F^{(-m)}_{<i}(x)$ is at
least $\rho$. Then the Oseledets subspace $E_i(x) = F^{(\infty)}_{\geq i}(x)
\cap F^{(-\infty)}_{\leq i}(x)$ is a well-defined $d_i$-dimensional space for
all $i \in I$. Moreover, the function $A_{20d\epsilon}(x)$ (defined
in~\eqref{eq:def_A_epsilon} in terms of these subspaces) satisfies
$A_{20d\epsilon}(x) \leq D$.
\end{thm}

Note that there is no randomness involved in this statement, it is completely
deterministic.

The condition on $B_\epsilon$ controls separately what happens in the past
and in the future. Oseledets subspaces are defined by intersecting flags
coming from the past and from the future, as explained in
Theorem~\ref{thm:Oseledets_limit}. Therefore, it is not surprising that there
should be an additional angle requirement to make sure that these flag
families are not too singular one with respect to the other. Note that the
angle requirement is expressed in terms of a fixed time $m$. Hence, it will
be easy to enforce in applications.

\bigskip

In this section, we prove Theorem~\ref{thm:deterministic_Oseledets}. Once and
for all, we fix $T$, $M$ and $\mu$ satisfying the assumptions of this
theorem, and constants $C>0$, $\epsilon\in (0, \min_{i\neq j \in I}
\abs{\lambda_i-\lambda_j}/(20d))$ and $\rho>0$. Consider a point $x$
satisfying $B_{\epsilon}(x) \leq C$. We want to show that, if $m$ is suitably
large (depending only on $C$, $\epsilon$ and $\rho$), then the subspaces
$F_{\geq i}^{(m)}(x)$ and $F_{<i}^{(-m)}(x)$ are well defined, and moreover
if the angle between them is at least $\rho$, then $A_{20d\epsilon}(x)$ is
bounded by a constant $D$ only depending on $C$, $\epsilon$ and $\rho$.

We will use the notations introduced before
Theorem~\ref{thm:Oseledets_limit}. In particular, $t_i^{(n)}(x) = e^{n
\lambda_i^{(n)}(x)}$ is the $i$-th singular value of $M^n(x)$. We will
essentially repeat the argument from the proof of a technical lemma
in~\cite{ruelle_pesin_theory}. A more detailed exposition is given in
Section~2.6.2 in~\cite{sarig_notes}.

\medskip

\emph{Step 1: there exists $N_1 = N_1(C, \epsilon)$ such that, if $n \geq
N_1$, then $\abs{\lambda_i^{(n)}(x) - \lambda_i} \leq 3\epsilon$ for all
$i$.} In particular, thanks to the inequality $\epsilon < \min_{i\neq j \in
I} \abs{\lambda_i-\lambda_j}/(20d)$, there is a gap between the eigenvalues
$\lambda_j^{(n)}(x)$ in different blocks $\{i,\dotsc, i+d_i-1\}$. (Note that
the $20d$ is much larger than what we need here, $6$ would be enough, but it
will be important later on.) This implies that the different subspaces
$(F_i^{(n)}(x))_{i\in I}$ are well defined.
\begin{proof}
We have $B_{\epsilon}(x) \leq C$. Thanks to the equality $\log
\norm{\Lambda^i M^n(x)} = n(\lambda_1^{(n)}(x)+\dotsc + \lambda_i^{(n)}(x))$,
and to the definition of $B_\epsilon^+$, this gives for all $i$
\begin{equation*}
  n \abs*{ \lambda_1^{(n)}(x)+\dotsc + \lambda_i^{(n)}(x) - (\lambda_1 + \dotsc+ \lambda_i)}
  \leq \epsilon n + C.
\end{equation*}
Subtracting these equations with indices $i$ and $i-1$, we get
$\abs{\lambda_i^{(n)}(x) - \lambda_i} \leq 2 \epsilon + 2C/n$. If $n$ is
large enough, this is bounded by $3\epsilon$ as desired.
\end{proof}

From this point on, we will only consider values of $n$ or $m$ which are
$\geq N_1$, so that the subspaces $F_i^{(n)}(x)$ are well defined. We will
write $\Pi^{(n)}_i$ for the orthogonal projection on this subspace, and
$\Pi^{(n)}_{\geq i}$ and $\Pi^{(n)}_{<i}$ for the projections on
$\bigoplus_{j \in I, j \geq i} F_j^{(n)}(x)$ and $\bigoplus_{j \in I, j < i}
F_j^{(n)}(x)$ respectively. They satisfy $\Pi^{(n)}_{\geq i} + \Pi^{(n)}_{<i}
= \Id$.

\medskip

\emph{Step 2: there exists a constant $K_1 = K_1(C, \epsilon)$ such that, for
all $m \geq n \geq N_1$, all $i>j$ in $I$ and all $v \in F_{\geq
i}^{(n)}(x)$, holds}
\begin{equation*}
  \norm{\Pi_{\leq j}^{(m)} v} \leq K \norm{v} e^{-n (\lambda_j - \lambda_i - 6(d-1)\epsilon)}.
\end{equation*}
\begin{proof}
The proof is done in two steps.

First claim: there exists a constant $K_0$ such that, for $n \geq N_1$, $v
\in F_{\geq i}^{(n)}(x)$ and $j<i$,
\begin{equation*}
  \norm{\Pi_j^{(n+1)} v} \leq K_0 \norm{v} e^{-n (\lambda_j - \lambda_i - 6\epsilon)}.
\end{equation*}

Indeed, on the one hand, we have
\begin{equation*}
  \norm{M^{n+1}(x) v} = \norm{M(T^n x) \cdot M^n(x) v} \leq (\sup_y \norm{M(y)}) \cdot \norm{M^n(x) v}
  \leq (\sup_y \norm{M(y)}) e^{n(\lambda_i+3\epsilon)} \norm{v},
\end{equation*}
thanks to the first step and the fact that $v \in F_{\geq i}^{(n)}(x)$. On
the other hand, as $M^{n+1}(x)$ respects the orthogonal decomposition into
the spaces $F_k^{(n+1)}(x)$, we have
\begin{equation*}
  \norm{M^{n+1}(x) v} \geq \norm{M^{n+1}(x) \Pi_j^{(n+1)} v}
  \geq e^{(n+1) (\lambda_j - 3\epsilon)} \norm{\Pi_j^{(n+1)} v},
\end{equation*}
again thanks to the first step. Putting these two equations together gives
the result.

\medskip

Second claim: for all $j<i$ in $I$, there exists a constant $K_{i,j}$ such
that, for all $m \geq n\geq N_1$ and all $v \in F_{\geq i}^{(n)}(x)$, we have
\begin{equation}
\label{eq:claim2}
  \norm{\Pi_{\leq j}^{(m)} v} \leq K_{i,j} e^{-n (\lambda_j-\lambda_i - 6(i-j) \epsilon)} \norm{v}.
\end{equation}
Once this equation is proved, then Step 2 follows by taking for $K_1$ the
maximum of the $K_{i,j}$ over $j<i$ in $I$. To prove~\eqref{eq:claim2}, we
argue by decreasing induction over $j < i$, $j \in I$. Assume thus that the
result is already proved for all $k \in I \cap (j,i)$, let us prove it for
$j$.

Decomposing a vector $v$ along its components on $F_{\leq j}^{(m)}(x)$, on
$F_k^{(m)}(x)$ for $k \in I \cap (j, i)$ and on $F_{\geq i}^{(m)}(x)$, we get
\begin{equation}
\label{eq:piqusfdp}
  \norm{\Pi_{\leq j}^{(m+1)} v} \leq \norm{\Pi_{\leq j}^{(m+1)} \Pi_{\leq j}^{(m)} v}
    + \sum_{k\in I \cap (j, i)} \norm{\Pi_{\leq j}^{(m+1)} \Pi_{k}^{(m)} v}
    + \norm{\Pi_{\leq j}^{(m+1)} \Pi_{\geq i}^{(m)} v}.
\end{equation}
The first term is bounded by $\norm{\Pi_{\leq j}^{(m)} v}$ as $\Pi_{\leq
j}^{(m+1)}$ is a projection. The second term is bounded by $K_0
e^{-m(\lambda_j-\lambda_k-6\epsilon)} \norm{\Pi_{k}^{(m)} v}$ thanks to the
first claim applied to $m$ and $\Pi_{k}^{(m)} v \in F_{\geq k}^{(m)}(x)$. The
induction hypothesis asserts that $\norm{\Pi_{k}^{(m)} v} \leq K_{k, i}
e^{-m(\lambda_k - \lambda_i -6(i-k)\epsilon)}\norm{v}$. Overall, we get for
the second term a bound which is at most
\begin{equation*}
  \sum_{k \in I \cap (j,i)} K_0 K_{i,k} e^{-m(\lambda_j - \lambda_i - 6(i-k + 1) \epsilon)}
  \leq K' e^{-m (\lambda_j -\lambda_i - 6(i-j) \epsilon)}\norm{v}.
\end{equation*}
Finally, the third term in~\eqref{eq:piqusfdp} is bounded by $K_0
e^{-m(\lambda_j - \lambda_i-6\epsilon)} \norm{\Pi_{\geq i}^{(m)} v}$, by the
first claim applied to $m$ and $\Pi_{\geq i}^{(m)} v \in F_{\geq
i}^{(m)}(x)$. This is bounded by $K_0 e^{-m(\lambda_j - \lambda_i-6\epsilon)}
\norm{v}$ as $\Pi_{\geq i}^{(m)}$ is a projection.

All in all, we have proved that
\begin{equation*}
  \norm{\Pi_{\leq j}^{(m+1)} v} \leq (K'+K_0) e^{-m (\lambda_j -\lambda_i - 6(i-j) \epsilon)}\norm{v}
  + \norm{\Pi_{\leq j}^{(m)} v}.
\end{equation*}
The estimate~\eqref{eq:claim2} then follows by induction over $m$, summing
the geometric series starting from $n$ as $\lambda_j -\lambda_i - 6(i-j)
\epsilon > 0$ thanks to the choice of $\epsilon$.
\end{proof}

The second step controls projections from $F_i^{(n)}$ to $F_j^{(m)}$, for $m
\geq n$, when $i > j$. The third step controls projections in the other
direction, thus giving a full control of the respective projections of the
spaces.

\medskip

\emph{Step 3: for all $m \geq n \geq N_1$, all $i>j$ in $I$ and all $v \in
F_{\leq j}^{(n)}$, holds}
\begin{equation*}
  \norm{\Pi_{\geq i}^{(m)} v} \leq K_1 \norm{v} e^{-n (\lambda_j - \lambda_i - 6(d-1)\epsilon)}.
\end{equation*}
\begin{proof}
Define a new matrix cocycle by $\tilde M(x) = (M^{-1}(x))^t$, from $E^*(x)$
to $E^*(T x)$. In coordinates (identifying $E(x)$ and $E^*(x)$ thanks to its
Euclidean structure), it is given as follows. Write $M^n(x)$ as $k_1 A k_2$
where $k_1$ and $k_2$ are orthogonal matrices, and $A$ is a diagonal matrix
with entries $t_1^{(n)}(x)=e^{n \lambda_1(n)(x)},\dotsc, t_d^{(n)}(x)=e^{n
\lambda_d(n)(x)}$. Then $\tilde M^n(x) = k_1 A^{-1} k_2$. Hence, it has the
same decomposition into singular spaces as $M^n(x)$, the difference being
that the singular values of $M^n(x)$ are replaced by their inverses.

The proof in Step 2 only used the fact that the logarithms of the singular
values were $3\epsilon$-close to $\lambda_i$, and the norm of the cocycle is
uniformly bounded. All these properties are shared by $\tilde M$. Hence, the
conclusion of Step 2 also applies to $\tilde M$, except that the inequality
between $i$ and $j$ have to be reversed as the ordering of singular values of
$\tilde M$ is the opposite of that of $M$. This is the desired conclusion.
\end{proof}

Overall, Steps 2 and 3 combined imply that the projection of a vector in
$F_i^{(n)}(x)$ on $(F_i^{(m)}(x))^\perp = F_{<i}^{(m)}(x) \oplus
F_{>i}^{(m)}(x)$ has a norm bounded by $2K_1 e^{-\delta n}$, for $\delta =
\min_{k \neq \ell \in I} \abs{\lambda_k -\lambda_\ell} -6(d-1) \epsilon > 0$.
Hence, in terms of the distance $\df$ on the Grassmannian of
$d_i$-dimensional subspaces defined in~\eqref{eq:distance_Grass}, we have
$\df(F_i^{(n)}(x), F_i^{(m)}(x)) \leq 2K_1 e^{-\delta n}$. It follows that
$F_i^{(n)}(x)$ is a Cauchy sequence, converging to a subspace
$F_i^{(\infty)}(x)$ as claimed in the statement of the theorem.

\medskip

\emph{Step 4: there exist $N_2 \geq N_1$ and a constant $K_2$ such that, for
all $n \geq N_2$, all $i$ in $I$ and all $v \in F_{i}^{(\infty)}$ with norm
$1$, holds}
\begin{equation}
\label{eq:step4}
  K_2^{-1} e^{n(\lambda_i -6d \epsilon)} \leq \norm{M^n(x) v} \leq K_2 e^{n (\lambda_i + 6d\epsilon)}.
\end{equation}
\begin{proof}
Take a vector $v \in F_i^{(\infty)}(x)$. For $j\in I$, the norm of the
projection $\pi_{F^{(n)}_j(x) \to F_i^{(\infty)}(x)}$, as the limit of the
projections $\pi_{F^{(n)}_j(x) \to F_i^{(m)}(x)}$, is bounded by $K_1 e^{-n
(\abs{\lambda_i -\lambda_j} - 6(d-1)\epsilon)}$ thanks to Steps 2 and 3 (note
that this bound is nontrivial only if $j \neq i$). Its transpose, the
projection $\pi_{F^{(\infty)}_i(x) \to F_j^{(n)}(x)}$, has the same norm and
therefore satisfies the same bound.

Writing $v_j = \pi_{F^{(\infty}_i(x) \to F_j^{(n)}(x)} v$, we have $M^n(x) v
= \sum_{j \in I} M^n(x) v_j$. We have
\begin{equation}
\label{eq:norm_vj}
  \norm{v_j} \leq  K_1 e^{-n (\abs{\lambda_i -\lambda_j} - 6(d-1)\epsilon)}.
\end{equation}
As $M^n(x)$ expands by at most $e^{n \lambda_j + 3\epsilon}$ on
$F_i^{(n)}(x)$, thanks to Step $1$, we obtain
\begin{equation*}
  \norm{M^n(x) v_j} \leq K_1 e^{-n (\abs{\lambda_i -\lambda_j} - 6(d-1)\epsilon)} e^{n \lambda_j + 3\epsilon}
  \leq K_1 e^{n (\lambda_i + 6d\epsilon)}.
\end{equation*}
Here, it is essential to have in Step 2 a control in terms of
$\lambda_j-\lambda_i$, and not merely some exponentially decaying term
without a control on the exponent. This proves the upper bound
in~\eqref{eq:step4}.

For the lower bound, we write $\norm{M^n(x) v} \geq \norm{M^n(x) v_i}$ as all
the vectors $M^{(n)}(x) v_j$ are orthogonal. This is bounded from below by
$e^{n (\lambda_i - 3\epsilon)} \norm{v_i}$, by Step 1. To conclude, it
suffices to show that $\norm{v_i}$ is bounded from below by a constant if $n$
is large enough. As $\norm{v_i} \geq \norm{v} - \sum_{j \neq i} \norm{v_j}$,
this follows from the fact that $\norm{v_j}$ tends to $0$ with $n$ if $j\neq
i$, thanks to~\eqref{eq:norm_vj}.
\end{proof}

We recall that we are trying to control the behavior of $M^n(x)$ not on
$F_i^{(\infty)}(x)$, but on the Oseledets subspace $E_i(x) = F_{\geq
i}^{(\infty)}(x) \cap F_{\leq i}^{(\infty)}(x)$. To this effect, there is in
the statement of Theorem~\ref{thm:deterministic_Oseledets} an additional
angle assumption that we will use now. Let $\rho>0$ be given as in the
statement of the theorem. There exists $\delta > 0$ with the following
property: if $U$ and $V$ are two subspaces of complementary dimension making
an angle at least $\rho$, then any subspaces $U'$ and $V'$ with $\df(U, U')
\leq \delta$ and $\df(V, V') \leq \delta$ make an angle at least $\rho/2$.

We fix once and for all $m_0=m_0(C, \epsilon, \delta) \geq N_2$ such that,
for all $i\in I$ and all $m\geq m_0$, one has $\df(F_{\geq i}^{(m)}(x),
F_{\geq i}^{(\infty)}(x)) \leq \delta$ and $\df(F_{\leq i}^{(-m)}(x), F_{\leq
i}^{(-\infty)}(x)) \leq \delta$. Its existence follows from the convergence
asserted at the end of Step 3 (and from the same result for $T^{-1}$).

Assume now (and until the end of the proof) that, for some $m\geq m_0$, the
angle between $F_{\geq i}^{(m)}(x)$ and $F_{<i}^{(-m)}(x)$ is $\geq \rho$, as
in the assumptions of the theorem. It follows then that the angle between
$F_{\geq i}^{(\infty)}(x)$ and $F_{< i}^{(-\infty)}(x)$ is at least $\rho/2$.
As a consequence, the spaces $F_{\geq i}^{(\infty)}(x)$ and $F_{\leq
i}^{(-\infty)}(x)$ are transverse, and their intersection is a
$d_i$-dimensional space $E_i(x)$.

\medskip

\emph{Step 5: there exist constants $K_3>0$ and $N_3 \geq N_2$ such that, for
all $n \geq N_3$, all $i\in I$ and all $v \in E_i(x)$ with norm $1$, holds}
\begin{equation}
\label{eq:step5}
  K_3^{-1} e^{n(\lambda_i -6d \epsilon)} \leq \norm{M^n(x) v} \leq K_3 e^{n (\lambda_i + 6d\epsilon)}.
\end{equation}
\begin{proof}
We have $v \in E_i(x) \subseteq F_{\geq i}^{(\infty)}(x)$. Decomposing the
vector $v$ along its components $v_j \in F_j^{(\infty)}(x)$ with $j\in I \cap
[i, d]$ and using the upper bound of~\eqref{eq:step4} for each $v_j$, the
upper bound in~\eqref{eq:step5} readily follows.

For the lower bound, we note that $E_i(x)$, being contained in $F_{\leq
i}^{(-\infty)}(x)$, makes an angle at least $\rho/2$ with
$F_{>i}^{(\infty)}(x)$. This implies that the norm of the projection $v_i$ of
$v$ on $F_i^{(\infty)}(x)$ is bounded from below, by a constant $c_0 > 0$.
Using both the upper and the lower bounds of Step 4, we obtain
\begin{equation*}
  \norm{M^n(x) v} \geq \norm{M^n(x) v_i} - \sum_{j \in I, j > i} \norm{M^n(x) v_j}
  \geq c_0 K_2^{-1} e^{n(\lambda_i -6d \epsilon)} - \sum_{j\in I, j > i} K_2 e^{n (\lambda_j + 6d\epsilon)}.
\end{equation*}
The choice of $\epsilon$ ensures that, for $j>i$ in $I$, one has $\lambda_i
-6d \epsilon > \lambda_j + 6d\epsilon$. Hence, the sum in this equation is
asymptotically negligible, and we obtain a lower bound $c_0 K_2^{-1}
e^{n(\lambda_i -6d \epsilon)} /2$ if $n$ is large enough.
\end{proof}

\emph{Step 6: there exists a constant $K_4$ such that, for all $n \in \Z$,
all $i\in I$ and all $v \in E_i(x)$ with norm $1$, holds}
\begin{equation}
\label{eq:step6}
  K_4^{-1} e^{n\lambda_i -6d \epsilon \abs{n}} \leq \norm{M^n(x) v}
  \leq K_4 e^{n \lambda_i + 6d\epsilon \abs{n}}.
\end{equation}
\begin{proof}
Step 5 shows that this control holds uniformly over $n \geq N_3$. The same
argument applied to the cocycle $M^{-1}$ and the map $T^{-1}$ gives the same
control for $n \leq -N_3$ (note that the function $B_\epsilon(x)$, which is
bounded by $C$ by assumption, controls both positive and negative times).
Finally, the control over $n \in (-N_3, N_3)$ follows from the finiteness of
this interval, and the uniform boundedness of $M$ and $M^{-1}$.
\end{proof}

We can finally conclude the proof of
Theorem~\ref{thm:deterministic_Oseledets}. We want to bound the quantity
$A_{20d \epsilon}(x)$ defined in~\eqref{eq:def_A_epsilon}. Fix $i\in I$, $v
\in E_i(x) \setminus \{0\}$ and $m, n \in \Z$. Then, using the upper bound
of~\eqref{eq:step6} for $\norm{M^n(x) v}$ and the lower bound for
$\norm{M^m(x) v}$, we get
\begin{align*}
  \frac{ \norm{M^n(x) v}}{\norm{M^m(x) v}} & e^{-(n-m)\lambda_i} e^{-(\abs{n} + \abs{m}) (20 d \epsilon)/2}
  \\& \leq K_4 e^{n \lambda_i + 6d \epsilon \abs{n}} \cdot K_4 e^{-m \lambda_i + 6d \epsilon \abs{m}} \cdot
       e^{-(n-m)\lambda_i} e^{-(\abs{n} + \abs{m}) (20 d \epsilon)/2}
  \\ &
  = K_4^2 e^{-(\abs{n} + \abs{m})4d \epsilon}
  \leq K_4^2.
\end{align*}
Taking the supremum over $i\in I$, $v \in E_i(x) \setminus \{0\}$ and $m, n
\in \Z$, this shows that $A_{20d \epsilon}(x) \leq K_4^2$. This concludes the
proof, for $D = K_4^2$. \qed

\section{Exponential returns to Pesin sets}

\label{sec:proof_returns_Pesin}

In this section, we prove Theorem~\ref{thm:exp_returns_Pesin}. As in the
assumptions of this theorem, let us consider a transitive subshift of finite
type $T$, with a Gibbs measure $\mu$ and a Hölder cocycle $M$ which has
exponential large deviations for all exponents. Let $\delta>0$. We wish to
show that, for some $D>0$, the set
\begin{equation*}
  \{x \st \Card\{k \in [0, n-1] \st A_\epsilon(T^k x) > D\} \geq \delta n\}
\end{equation*}
has exponentially small measure. Reducing $\epsilon$ if necessary, we can
assume $\epsilon < \abs{\lambda_i-\lambda_j}$ for all $i\neq j \in I$. Set
$\epsilon' = \epsilon/(20d)$.

The angle between the Lyapunov subspaces is almost everywhere nonzero. In
particular, given $i\in I$, the angle between $F_{\geq i}^{(\infty)}(x)$ and
$F_{<i}^{(-\infty)}(x)$ is positive almost everywhere. On a set of measure
$>1-\delta/2$, it is bounded from below by a constant $2 \rho>0$ for all $i$.
These subspaces are the almost sure limit of $F^{(m)}_{\geq i}(x)$ and
$F^{(-m)}_{<i}(x)$, according to Theorem~\ref{thm:Oseledets_limit}. Hence, if
$m$ is large enough, say $m \geq m_1$, the set
\begin{multline*}
  U = U_m = \{x\in X \st \forall i \in I, F^{(m)}_{\geq i}(x) \text{ and }F^{(-m)}_{<i}(x)\text{ are well defined} \\
        \text{and } \angle(F^{(m)}_{\geq i}(x),F^{(-m)}_{<i}(x)) > \rho\}
\end{multline*}
has measure $>1-\delta/2$.

We will use the functions $B^{(i)\pm}_{\epsilon'}$ defined
before~\eqref{eq:def_B_epsilon}.  For each $i \in [1,d]$ and $\sigma \in
\{+,-\}$, there exists a constant $C_{i,\sigma}$ such that
\begin{equation*}
  \{x \st \Card\{k \in [0, n-1] \st B^{(i)\sigma}_{\epsilon'}(T^k x) > C_{i,\sigma}\} \geq \delta n/ (4d)\}
\end{equation*}
has exponentially small measure, by Theorem~\ref{thm:exp_returns} and the
assumption on exponential large deviations for all exponents. (For
$\sigma=-$, this theorem should be applied to $T^{-1}$). Let $C' = \max
C_{i,\sigma}$. As $B_{\epsilon'}$ is the maximum of the functions
$B^{(i)\sigma}_{\epsilon'}$, it follows that
\begin{equation*}
  \{x \st \Card\{k \in [0, n-1] \st B_{\epsilon'}(T^k x) > C' \} \geq \delta n/ 2\}
\end{equation*}
has exponentially small measure.

We apply Theorem~\ref{thm:deterministic_Oseledets} with $\epsilon=\epsilon'$
and $C=C'$ and $\rho$, obtaining some integer $m_0 \geq 1$ and some constant
$D$ with the properties described in Theorem 5.1. Let us fix until the end of
the proof $m = \max(m_0, m_1)$.

The set $U = U_m$ is open by continuity of $M^m$ and $M^{-m}$. In particular,
it contains a set $V$ which is a finite union of cylinders, with
$\mu(V)>1-\delta/2$. To conclude, it suffices to show that
\begin{equation}
\label{eq:piouqsif}
 \{x \st \Card\{k \in [0, n-1] \st T^k x \notin V\} \geq \delta n/ 2\}
\end{equation}
has exponentially small measure. Indeed, assume this holds. Then, apart from
an exponentially small set, there are at most $\delta n$ bad times $k$ in
$[0,n-1]$ for which $T^k x \notin V$ or $B_{\epsilon'}(T^k x) > C'$. For the
other good times, we have $T^k x \in V$ and $B_{\epsilon'}(T^k x) \leq C'$.
Then Theorem~\ref{thm:deterministic_Oseledets} shows that $A_\epsilon(T^k x)
= A_{20d\epsilon'}(T^k x)\leq D$, as desired.

It remains to control~\eqref{eq:piouqsif}. Let $\chi_{V}$ denote the
characteristic function of $V$, it is a continuous function. The set
in~\eqref{eq:piouqsif} is
\begin{equation*}
  \{x \st S_n \chi_{V} (x) < (1-\delta/2) n\}.
\end{equation*}
As $\int \chi_{V} = \mu(V) > 1-\delta/2$ by construction, the large deviation
principle for continuous functions shows that this set is indeed
exponentially small. This concludes the proof of the theorem. \qed

\appendix

\section{Counterexamples to exponential large deviations}

\label{app:counter}

In this appendix, we give two counterexamples to exponential large
deviations. The first easy one, in Proposition~\ref{prop:counter_subadd}, is
for Hölder-continuous subadditive cocycles. The second harder one, in
Theorem~\ref{thm:counter_example}, is in the more restrictive setting of
norms of matrix cocycles (only continuous, although one expects that the same
kind of result should hold for Hölder cocycles with small Hölder exponent).

\begin{prop}
\label{prop:counter_subadd} Let $(T,\mu)$ be an invertible subshift with an
invariant ergodic measure $\mu$ which is not supported on a periodic orbit.
Consider a positive sequence $u_n$ tending to $0$. There exists a subadditive
cocycle $a(n,x)$ such that $a(n,\cdot)$ is Hölder continuous for any $n$,
such that $a(n,x)/n \to 0$ almost everywhere, and such that, for infinitely
many values of $n$,
\begin{equation*}
  \mu\{  x \st a(n,x)/n \leq -1\} \geq u_n.
\end{equation*}
\end{prop}

The proof uses the following easy variant of Rokhlin's lemma:
\begin{lem}
\label{lem:rokhlin} Let $\delta>0$ and $m>0$. In a subshift in which the set
of periodic points has measure $0$, there exists a subset $R$ made of
finitely many cylinders such that the sets $(T^i R)_{0\leq i < m}$ are
pairwise disjoint and cover a measure at least $1-\delta$.
\end{lem}
\begin{proof}
We may find a set $S$ such that its $m$ first iterates are disjoint and cover
a measure $\geq 1-\delta/2$, by Rokhlin's lemma. Let $S'$ be a finite union
of cylinders which approximates $S$ so well that $\mu(S' \Delta S) \leq
\rho$, for $\rho=\delta/(4m^2)$. Let $R = S' \setminus \bigcup_{0<i<m}
T^i(S')$. It is a finite union of cylinder sets, and the sets $T^i R$ for
$i<m$ are disjoint. We have
\begin{equation*}
  S' \cap T^i(S') \subseteq (S'\Delta S) \cup (S \cap T^i S) \cup (T^i S \Delta T^i S').
\end{equation*}
The middle set is empty, the other ones have measure at most $\rho$. Hence,
the measure of this set is at most $2\rho$. Finally, $\mu(R) \geq \mu(S') -
2(m-1) \rho \geq \mu(S) - 2m \rho$. Hence.
\begin{equation*}
  \mu\left(\bigcup_{0\leq i<m} T^i R\right) =
  m \mu(R) \geq m \mu(S) - 2m^2 \rho = \mu\left(\bigcup_{0\leq i<m} T^i S\right) - 2 m^2 \rho
  \geq 1 -\delta/2 - 2 m^2 \rho.
\end{equation*}
The choice of $\rho$ ensures that the last term is $1-\delta$, as claimed.
\end{proof}

\begin{proof}[Proof of Proposition~\ref{prop:counter_subadd}]
We will construct a sequence $n_i \to \infty$ and a sequence of functions
$f_i$ for $i\geq 1$ with the following properties:
\begin{enumerate}
\item Each $f_i$ is Hölder continuous (in fact, it will only depend on
    finitely many coordinates).
\item We have $f_i(x) \leq 0$ for all $x$, and $\int f_i = -2^{-i}$.
\item We have $\mu\{x \st S_{n_i} f_i(x) \leq -2 n_i\} \geq u_{n_i}$.
\end{enumerate}
Let also $f_0=1$ and $n_0=0$. Define then
\begin{equation*}
  a(n,x) = \sum_{i\st n_i \leq n} S_n f_i(x).
\end{equation*}
As the $(f_i)_{i\geq 1}$ are nonpositive, this is a subadditive cocycle.
Moreover, $\int a(n,x)/n = \sum_{n_i \leq n} \int f_i \to 0$. By Kingman's
theorem, it follows that $a(n,x)/n$ tends to $0$ almost surely. Moreover, if
$S_{n_i} f_i(x) \leq -2 n_i$, then by nonpositivity of all the $f_j$ except
for $j=0$,
\begin{equation*}
  a(n_i,x) \leq S_{n_i} f_0(x) + S_{n_i} f_i(x) \leq n_i -2n_i \leq -n_i.
\end{equation*}
Hence, the third point in the definition of $f_i$ ensures that $a(n_i,x) \leq
-n_i$ with probability at least $u_{n_i}$, showing that $a$ satisfies the
conclusion of the proposition.

Let us now construct $f_i$ and $n_i$ as above. First, choose $n=n_i$ such
that $u_{n_i} \leq 2^{-i-3}$. Then, let $K=2^{i+2} n_i$. We use a
corresponding Rokhlin tower: by Lemma~\ref{lem:rokhlin}, there exists a set
$R$ which is a finite union of cylinder sets such that $R,\dotsc, T^{K-1}R$
are disjoint, and their union covers a proportion $>1/2$ of the space. Then
$\mu(R) \in (1/(2K), 1/K]$. Define $f_i$ to be equal to $-c_i$ on
$\bigcup_{k<K/2^{i+1}} T^k R$ and $0$ elsewhere, where $c_i$ is chosen so
that $\int f_i = -2^{-i}$. As $\mu(\bigcup_{k<K/2^{i+1}} T^k R) = (K/2^{i+1})
\mu(R) \leq 2^{-i-1}$, it satisfies $c_i \geq 2$. For any $x \in
\bigcup_{k<K/2^{i+2}} T^k R$, one has $f_i(T^k x) = -c_i$ for $k <
K/2^{i+2}=n_i$, and therefore $S_{n_i} f_i(x) = -c_i n_i \leq -2 n_i$. The
probability of this event is $\mu\left(\bigcup_{k<K/2^{i+2}} T^k R\right) =
(K/2^{i+2}) \mu(R) \geq 2^{-i-3} \geq u_{n_i}$, as desired.
\end{proof}

We will now construct a continuous cocycle taking values in $SL(2,\R)$
without exponential large deviations for its top exponent. Note that a
generic continuous cocycle away from uniform hyperbolicity has only zero
Lyapunov exponents, by Bochi-Viana~\cite{bochi_viana}, so it has exponential
large deviations by Theorem~\ref{thm:large_deviations} (1). Hence, our
construction can not be done using Baire arguments.

\begin{thm}
\label{thm:counter_example} Let $u_n$ be any positive sequence tending to
$0$. Consider the full shift on two symbols with a fully supported invariant
ergodic measure $\mu$. Then there exists a continuous $SL(2,\R)$-valued
cocycle $M$ with a positive top Lyapunov $\lambda_+(M)$ such that, for
infinitely many values of $n$,
\begin{equation*}
  \mu\{x \st \log \norm{M^n(x)} \leq n \lambda_+(M)/2\} \geq u_n.
\end{equation*}
\end{thm}
If $u_n$ tends to zero slower than exponentially, for instance $u_n=1/n$,
then the cocycle $M$ does not have exponential large deviations.

\bigskip

Let $\Sigma$ be the full shift over two symbols $0$ and $1$, with a given
invariant ergodic measure $\mu$ of full support (what we really need is that
the support of $\mu$ contains a fixed point, or more generally a periodic
orbit, but $\mu$ is not supported on this orbit). In this section, we will
say that an object defined on $\Sigma$ is locally constant if it only depends
on $(x_n)_{\abs{n}\leq N}$ for some $N$. Let $x_* \in \Sigma$ be the point
with all coordinates equal to $1$. We say that a cocycle $M$ taking values in
$SL(2,\R)$ has property $P_\lambda$, for some $\lambda>0$, if it satisfies
the following properties:
\begin{enumerate}
\item The cocycle $M$ is locally constant.
\item Its largest Lyapunov exponent is $>\lambda$.
\item Its Oseledets subspaces, initially defined $\mu$-almost everywhere,
    are in fact locally constant (and therefore continuous).
\item One has $M(x_*)=\Id$.
\end{enumerate}
Define a cocycle $M_0$ by $M_0(x)=\left(\begin{smallmatrix} 2 & 0 \\
0 & 1/2
\end{smallmatrix}\right)$ if $x_0=0$, and $M_0(x) = \Id$ if $x_0=1$. Then its Oseledets subspaces are $\R
\oplus \{0\}$ and $\{0\} \oplus \R$, and the corresponding Lyapunov exponents
are nonzero. Hence, $M_0$ satisfies $P_\lambda$ for some $\lambda>0$.

The main lemma is the following:
\begin{lem}
\label{lem:not_exp_induc} Let $\lambda>0$ and $\epsilon>0$ and $n_0>0$. Let
$M$ be a cocycle with the property $P_\lambda$. Then there exist a time
$n>n_0$ and another cocycle $\tilde M$, again having the property
$P_\lambda$, with the following properties:
\begin{enumerate}
\item For all $x$, one has $\norm{\tilde M(x) - M(x)} \leq \epsilon$.
\item There exists a set $A$ with measure $\geq u_n$ on which $\norm{\tilde
    M^n(x)} < e^{\lambda n/2}$.
\end{enumerate}
\end{lem}

Let us admit the lemma for the time being. We construct inductively a
sequence of cocycles $M_i$, all with the property $P_\lambda$, starting with
$M_0$ as above. Suppose that we have already constructed times $n_1,\dotsc,
n_{i-1}$, sets $A_1,\dotsc, A_{i-1}$ with $\mu(A_j) \geq u_{n_j}$, and the
cocycle $M_{i-1}$ such that, for each $j<i$, $\norm{M_{i-1}^{n_j}(x)} <
e^{\lambda n_j/2}$ for all $x\in A_j$. We wish to construct a time
$n_i>n_{i-1}$, a set $A_i$ and a cocycle $M_i$ that satisfies the same
properties for all $j\leq i$. Note that, if $\epsilon=\epsilon_i$ is small
enough, then any cocycle $M_i$ with $\norm{M_i(x) - M_{i-1}(x)} \leq
\epsilon$ for all $x$ will satisfy the above properties for $j<i$, with the
same sets $A_j$. Hence, it suffices to apply Lemma~\ref{lem:not_exp_induc} to
$M = M_{i-1}$, with a sufficiently small $\epsilon$, to get $M_i = \tilde M$.

We can require $\epsilon_i \leq 1/2^i$. Then the sequence $M_i$ converges
uniformly, towards a limiting continuous cocycle $M(x)$. By semi-continuity
of the Lyapunov exponents, $\lambda_+(M) \geq \limsup \lambda_+(M_i) \geq
\lambda$. On the other hand, $\norm{M^{n_j}(x)} \leq e^{\lambda n_j/2}$ for
all $x\in A_j$, and this set has measure at least $u_{n_j}$ as claimed. This
concludes the proof of Theorem~\ref{thm:counter_example}. \qed

\medskip

It remains to prove Lemma~\ref{lem:not_exp_induc}. The main tool to modify
the cocycle is the following lemma, due to Bochi.

\begin{lem}
\label{lem:exchange} Assume that the cocycle $M$ satisfies $P_\lambda$. Let
$\epsilon>0$. Then, for almost every $x$, there exist $k(x) \in \N$ and
matrices $Q_0, \dotsc, Q_{k-1}$ such that $\norm{Q_i - M(T^i x)} \leq
\epsilon$ for all $i<k$, and the product $Q_{k-1} \dotsb Q_0$ sends $E^u(x)$
to $E^s(T^k x)$, and $E^s(x)$ to $E^u(T^k x)$ (where $E^s$ and $E^u$ are the
stable and unstable Oseledets directions of the cocycle $M$).
\end{lem}
\begin{proof}
The set $A$ of points that satisfy the conclusion of the lemma is backwards
invariant under the dynamics: if $Tx=y$ and the sequence of matrices
$Q_0,\dotsc,Q_{k-1}$ works for $y$, then the sequence of matrices
$\Id,Q_0,\dotsc, Q_{k-1}$ works for $x$, for $k(x) = k(y)+1$. By ergodicity,
it suffices to show that $A$ has positive measure. This follows
from~\cite[Proposition~9.10]{viana_lyapunov}, as the cocycle $M$ is not
uniformly hyperbolic thanks to the condition $M(x_*)=\Id$ in $P(4)$. (In our
case, there is a direct easy proof as the cocycle is the identity on a
neighborhood of the fixed point $x_*$, so it can be replaced by a small
rotation in suitable coordinates, on points whose orbit spends a long enough
time close to $x_*$).
\end{proof}

\begin{proof}[Proof of Lemma~\ref{lem:not_exp_induc}]
The idea is to apply Lemma~\ref{lem:exchange} at some points, modifying the
cocycle along a piece of orbit of length $k$, and then again the same lemma
$n$ steps later (for some $n$ much larger than $k$), to put again $E^s$ in
line with $E^s$, and $E^u$ in line with $E^u$. The norm of the new cocycle
will essentially not increase along these $n$ steps thanks to the
cancellations between the stable and unstable directions, yielding the
desired set $A$, while the Lyapunov exponent will essentially not be changed
if these $n$ steps are negligible compared to the whole dynamics. Making this
precise requires the use of the Rokhlin tower provided by
Lemma~\ref{lem:rokhlin}, and some care when choosing the constants.

The cocycle $M$ and its Oseledets subspaces are constant on cylinders of
length $2N+1$, for some $N$, by assumption. Replacing the original subshift
by a new subshift the symbols of which correspond to $2N+1$-cylinders of the
original subshift, we may assume without loss of generality that $N=0$, i.e.,
the cocycle $M(x)$ and the Oseledets subspaces $E^s(x)$ and $E^u(x)$ only
depend on the coordinate $x_0$ of $x$.

The minimal function $k(x)$ provided by Lemma~\ref{lem:exchange} is
measurable. Hence, it is bounded on a set of arbitrarily large measure. We
obtain an integer $k>0$, a set $X$ with $\mu(X) > 9/10$, and for each $x\in
X$ a sequence of matrices $Q_0(x),\dotsc, Q_{k-1}(x)$ with
\begin{equation}
  \label{normQi}
  \norm{Q_i(x) - M(T^i x)} \leq \epsilon
\end{equation}
whose product $Q_{k-1}(x) \dotsm Q_0(x)$ maps $E^s(x)$ to $E^u(T^k x)$ and
$E^u(x)$ to $E^s(T^k x)$.

Let $\lambda_+(M)>\lambda$ be the top Lyapunov exponent of $M$. Let
$\delta>0$ be small enough so that $14\delta < \lambda$. For $\mu$-almost
every $x$, there exists a constant $C(x)<\infty$ such that, for all $\ell\in
\Z$
\begin{equation*}
  C(x)^{-1} e^{-\delta \abs{\ell}} \leq \frac{\norm{M^\ell(x) v^u(x)}}{e^{\lambda_+(M) \ell}} \leq C(x) e^{\delta \abs{\ell}},
  \quad
  C(x)^{-1} e^{-\delta \abs{\ell}} \leq \frac{\norm{M^\ell(x) v^s(x)}}{e^{-\lambda_+(M) \ell}} \leq C(x) e^{\delta \abs{\ell}},
\end{equation*}
where $v^u(x)$ and $v^s(x)$ are unit vectors in $E^u(x)$ and $E^s(x)$.
Shrinking $X$ just a little bit, we can assume that $C(x)$ is bounded by a
constant $C_0$ on $X$, while retaining the estimate $\mu(X) > 9/10$.

As the Oseledets subspaces depend continuously on the point, by
$P_\lambda(3)$, the angle between $v^u(x)$ and $v^s(x)$ is bounded from
below. Hence, increasing $C_0$ if necessary, we can ensure that, for any
matrix $A$ and any $x$,
\begin{equation}
\label{eq:norm_bound_vu_vs}
  \norm{A} \leq C_0 \max(\norm{A v^u(x)}, \norm{A v^s(x)}).
\end{equation}
Increasing $C_0$ and shrinking $X$ if necessary, we can also assume that, for
any $x\in X$, the global modification matrix at $x$ given by $\tilde Q(x) =
M^k(x)^{-1}Q_{k-1}(x) \dotsm Q_0(x)$ (which exchanges $E^u(x)$ and $E^s(x)$)
expands all vectors by at most $C_0$, and contracts them by at most
$C_0^{-1}$.

\bigskip

Let $n\geq k$ be such that $C_0 \leq e^{\delta n}$. Let $m=Kn$, where $K\geq
6$ will be chosen later, independently of $n$. Applying
Lemma~\ref{lem:rokhlin}, we obtain a set $R$ which is a finite union of
cylinders, whose first $m$ iterates are disjoint and cover a measure $>
9/10$. Subdividing $R$ further if necessary, we may write it as a disjoint
union of cylinders $R_p$ of length $2r+1$, centered around $0$, for some
$r\geq m+k$. Let $O_p = \bigcup_{i<m} T^i R_p$, these sets are disjoint. We
will make the modifications of the cocycle separately on each $O_p$.

The point $x_*$ is in at most one $O_p$. If it belongs to $O_1$, say, then we
remove $R_1$ from $R$. Increasing $r$ if necessary, this removes an
arbitrarily small measure from $R$, so the new $R$ will still satisfy the
condition $\mu(\bigcup_{j<m} T^j R) > 9/10$. This means that modifying the
cocycle on the sets $O_p$ will not change its value on $x_*$, so that the
condition $M(x_*)=\Id$ in $P_\lambda(4)$ will still be satisfied by the
modified cocycle.

We say that a set $O_p$ is modifiable if there exists an index $a_p\in [0,
m-3n)$ such that $T^{a_p} R_p$ intersects $X \cap T^{-n} X$. If $O_p$ is not
modifiable, then the set $\tilde O_p = \bigcup_{a<m-3n} T^a R_p$ (whose
measure is at least $\mu(O_p)/2$ as $m-3n \geq m/2$) does not intersect $X
\cap T^{-n} X$. Hence, the union of these $\tilde O_p$ has measure at most
$1-\mu(X \cap T^{-n} X) \leq 2/10$, the union of the corresponding $O_p$ has
measure at most $4/10$, and the measure of the union of the modifiable $O_p$
is at least $9/10 - 4/10 = 1/2$.

Let $O_p$ be modifiable. Choose a point $x_p \in T^{a_p} R_p \cap X \cap
T^{-n} X$. On $O_p$, we define the cocycle $\tilde M$ to be equal to
$Q_i(x_p)$ on $T^{a_p+i} R_p$ for $0\leq i < k$, to $Q_i(T^n x_p)$ on
$T^{a_p+n+i} R_p$ for $0\leq i < k$, and to $M$ elsewhere. The cocycle $M$ is
constant on each set $T^i R_p$ (as $M(x)$ only depends on $x_0$, and $R_p$ is
a cylinder of length $2r+1$ with $r>m$). Hence, it follows
from~\eqref{normQi} that $\norm{\tilde M(x)-M(x)} \leq \epsilon$ everywhere.
Moreover, it is clear from the construction that $\tilde M$ is locally
constant.

Let us show that the Lyapunov exponent of $\tilde M$ is $>\lambda$. Start
from a point $x$ which is not in the modified locus $\bigcup_p \bigcup_{a_p
\leq i < a_p+n+k} T^i R_p$, we will estimate the expansion of $\tilde
M^\ell(x)v^u(x)$ when $\ell$ tends to $\infty$. Except when $T^\ell x$
belongs to the modified locus, the vector $\tilde M^\ell(x)v^u(x)$ is a
multiple of $v^u(T^\ell x)$, and undergoes the same expansion under $M$ or
$\tilde M$. The difference is the influence of the modified locus: when one
enters this locus, then one should apply the modification operator $\tilde
Q(x_p)$ which brings $v^u(x_p)$ to $v^s(x_p)$ (with an expansion at least
$C_0^{-1}$), then the original cocycle $M^n(x_p)$ but on the vector
$v^s(x_p)$, then the modification operator $\tilde Q(T^n x_p)$ that brings
back $v^s(T^n x_p)$ to $v^u(T^n x_p)$ (again with an expansion at least
$C_0^{-1}$). Then, one follows again the dynamics of the cocycle $M$. During
such a visit to the modified locus, the expansion under $\tilde M$ is at
least $C_0^{-1} \cdot C_0^{-1}e^{-\lambda_+(M) n -\delta n} \cdot C_0^{-1}$,
while the expansion under $M$ is at most $C_0 e^{\lambda_+(M) n + \delta n}$.
Hence, the expansion loss for $\tilde M$ with respect to $M$ is at most
$C_0^{-4} e^{-2\lambda_+(M) n -2\delta n} \geq e^{-2\lambda_+(M) n - 6 \delta
n}$. Moreover, such a loss happens at most once in every $m$ steps, since a
visit to $O_p$ has length $m$ by construction. We get
\begin{equation*}
  \lambda_+(\tilde M) \geq \lambda_+(M) - (2\lambda_+(M) + 6\delta) n/m.
\end{equation*}
By assumption, $\lambda_+(M)>\lambda$. If the ratio $K=m/n$ is large enough,
it follows that one also has $\lambda_+(\tilde M) > \lambda$.

The same argument shows that, towards the past, $v^u(x)$ is exponentially
contracted. Hence, $v^u(x)$ generates the Oseledets subspace $E^u(x)$ for
$\tilde M$. This shows that, away from the modified locus, the Oseledets
subspace is locally constant. Using its equivariance under $\tilde M$ and the
fact that $\tilde M$ is locally constant, we deduce that the Oseledets
subspace of $\tilde M$ is locally constant everywhere.

We have proved that $\tilde M$ satisfies $P_\lambda$. It remains to show the
existence of a set $A$ with measure $\geq u_n$ on which $\norm{\tilde M^n(x)}
< e^{n \lambda/2}$. We take for $A$ the union of the sets $T^{a_p + i} R_p$
over $i \in [n/2-\delta n/\lambda_+(M), n/2+\delta n/\lambda_+(M)]$ and $p$
such that $O_p$ is modifiable. In each modifiable set, $A$ takes a proportion
$(2\delta n/\lambda_+(M))/ m = 2\delta/(K \lambda_+(M))$. As the measure of
modifiable sets $O_p$ is at least $1/2$, we get $\mu(A) \geq \delta/(K
\lambda_+(M))$, a number which is independent of $n$. In particular, if $n$
is large enough, we get $\mu(A) \geq u_n$ as $u_n$ tends to $0$ with $n$.

Consider $x\in A$, let us show that $\norm{\tilde M^n(x)} < e^{n \lambda/2}$
to conclude the proof. Consider $p$ and $i=n/2+j$ with $\abs{j} \leq \delta
n/\lambda_+(M)$ such that $x \in T^{a_p + i} R_p$. First, we estimate the
norm of $\tilde M^n(x) v^s(x)$. This vector is obtained by iterating the
original cocycle $M$ during $n-i$ steps, then doing the modification $\tilde
Q(T^n x_p)$ that brings it to $v^u(T^n x_p)$, and then iterating the original
cocycle $M$ during $i$ steps. The first step results in an expansion by at
most $C_0 e^{-\lambda_+(M) (n-i) + \delta(n-i)}$ (as $T^n x_p \in X$), the
second one by an expansion at most $C_0$, and the third one by an expansion
at most $C_0 e^{\lambda_+(M) i + \delta i}$. In the end, we obtain
\begin{equation*}
  \norm{\tilde M^n(x) v^s(x)} \leq C_0^3 e^{\delta n} e^{-\lambda_+(M)(n-i) + \lambda_+(M) i}
  \leq C_0^3 e^{\delta n} e^{2\lambda_+(M) \abs{j}} \leq C_0^3 e^{3\delta n}.
\end{equation*}
In the same way, $v^u(x)$ is expanded by at most $C_0 e^{\lambda_+(M)(n-i) +
\delta(n-i)}$ during the first $n-i$ iterates, then by at most $C_0$ by the
modification $\tilde Q(T^n x_p)$ that brings it to $v^s(T^n x_p)$, and then
by at most $C_0 e^{-\lambda_+(M) i + \delta i}$ for the last $i$ iterates.
Hence,
\begin{equation*}
  \norm{\tilde M^n(x) v^u(x)} \leq C_0^3 e^{\delta n} e^{\lambda_+(M)(n-i) - \lambda_+(M) i}
  \leq C_0^3 e^{3\delta n}.
\end{equation*}
With~\eqref{eq:norm_bound_vu_vs}, this gives
\begin{equation*}
  \norm{\tilde M^n(x)} \leq C_0^4 e^{3\delta n} \leq e^{7\delta n} < e^{n\lambda/2},
\end{equation*}
thanks to the choice of $\delta$.
\end{proof}

\bibliography{biblio}
\bibliographystyle{amsalpha}

\end{document}